\documentclass[3p,twocolumn]{elsarticle}

\usepackage{cite}
\usepackage{amsmath,amssymb,amsfonts}
\usepackage{amsthm}
\usepackage{color}
\usepackage{tikz}
\usepackage{pifont}
\usepackage{interval}

\newtheorem{theorem}{Theorem}[section]

\newtheorem{definition}[theorem]{Definition}
\newtheorem{corollary}[theorem]{Corollary}
\newtheorem{proposition}[theorem]{Proposition}

\newtheorem{remark}[theorem]{Remark}

\newtheorem{example}[theorem]{Example}

\newcommand{\cmark}{\ding{51}}%
\newcommand{\xmark}{\ding{55}}%

\newcommand{\setdef}[2]{\left\{#1 \; \mid \; #2\right\}}

\newcommand{\until}[1]{\{1,\dots,#1\}}
\newcommand{\fromto}[2]{\{#1,\dots,#2\}}

\newcommand{\Cc}{\mathcal{C}}
\newcommand{\Dc}{\mathcal{D}}

\newcommand{\Fc}{\mathcal{F}}

\newcommand{\Ic}{\mathcal{I}}

\newcommand{\Kc}{\mathcal{K}}
\newcommand{\Lc}{\mathcal{L}}

\newcommand{\Nc}{\mathcal{N}}

\newcommand{\Pc}{\mathcal{P}}

\newcommand{\Uc}{\mathcal{U}}
\newcommand{\Vc}{\mathcal{V}}

\newcommand{\Xc}{\mathcal{X}}
\newcommand{\Yc}{\mathcal{Y}}

\newcommand{\real}{\mathbb{R}}

\intervalconfig {
  soft open fences ,
  separator symbol = ,
}

\DeclareSymbolFont{bbold}{U}{bbold}{m}{n}
\DeclareSymbolFontAlphabet{\mathbbold}{bbold}

\newcommand{\norm}[1]{\left\lVert#1\right\rVert}

\newcommand{\proj}{\operatorname{proj}}

\newcommand\oprocendsymbol{\hbox{$\bullet$}}
\newcommand\oprocend{\relax\ifmmode\else\unskip\hfill\fi\oprocendsymbol}

\newcommand*{\QEDA}{\hfill\ensuremath{\blacksquare}}%

\newcommand\xqed[1]{%
  \leavevmode\unskip\penalty9999 \hbox{}\nobreak\hfill
  \quad\hbox{#1}}
\newcommand\demo{\xqed{$\bullet$}}

\newcommand{\longthmtitle}[1]{\mbox{}\emph{(#1):}}

\begin{document}

\begin{frontmatter}
  
  \title{\bf Regularity Properties of Optimization-Based Controllers}
  
  \author[1]{Pol Mestres\corref{cor1}}
  \ead{pomestre@ucsd.edu}
  \author[2]{Ahmed Allibhoy}
  \ead{aallibho@ucr.edu}
  \author[1]{Jorge Cort{\'e}s}
  \ead{cortes@ucsd.edu}
  \cortext[cor1]{Corresponding author}
  \address[1]{Department of Mechanical and Aerospace Engineering,
    University of California, San Diego, 9500 Gilman Dr, La Jolla, CA 92093}
  
    \address[2]{Department of Mechanical Engineering,
    University of California, Riverside, 900 University Ave., Riverside, CA 92521}

  
  \begin{keyword}
    Parametric optimization, optimization-based control, existence and
    uniqueness of solutions
  \end{keyword} 
  
  \begin{abstract}
    This paper studies regularity properties of optimization-based
    controllers, which are obtained by solving optimization problems
    where the parameter is the system state and the optimization
    variable is the input to the system.  Under a wide range of
    assumptions on the optimization problem data, we provide an
    exhaustive collection of results about their regularity, and
    examine their implications on the existence and uniqueness of
    solutions and the forward invariance guarantees for the resulting
    closed-loop systems.  We discuss the broad relevance of the
    results in different areas of systems and controls.
  \end{abstract}
\end{frontmatter}

\section{Introduction}\label{sec:introduction}

Optimization-based controllers are ubiquitous in numerous areas of
systems and control including safety-critical
control~\citep{ADA-SC-ME-GN-KS-PT:19}, model predictive
control~\citep{GEC-DMP-MM:89,JBR-DQM-MMD:17}, and online feedback
optimization~\citep{MC-EDA-AB:20,AH-SB-GH-FD:21}.  Optimization-based
controllers are a particular class of parametric optimization
problems. The theory of parametric
optimization~\citep{AVF-GPM:90,GS:18,JFB-AS:98} considers optimization
problems that depend on a parameter and studies the regularity
properties of the minimizers with respect to it. In optimization-based
control, the parameter is the state and the optimization variable is
the input.

Given a dynamical system (either in discrete or continuous time) with
state $x \in \real^n$ and input $u \in \real^m$, an optimization-based
controller is a feedback law obtained by solving a problem of the form
\begin{subequations}\label{eq:optim-based-controller}
  \begin{align}
    &   \underset{u\in\real^m}{\text{argmin}} \
    f(x,u)
  \\
  & \qquad \text{s.t.} \ g(x,u) \leq 0
\end{align}  
\end{subequations}
with $f:\real^n\times\real^m\to\real$ and
$g:\real^n\times\real^m\to\real^p$.  Note that the system state $x$
acts as a parameter in~\eqref{eq:optim-based-controller}.  Assuming
that the optimizer of~\eqref{eq:optim-based-controller} is unique for
every $x\in\real^n$, this defines a function $u^*:\real^n\to\real^m$,
mapping each state to the optimizer
of~\eqref{eq:optim-based-controller}.  This approach allows to encode
desirable goals for controller synthesis both in the cost function~$f$
and in the constraints~$g$. For instance, desirable performance
objectives such as minimum control effort or maximizing convergence
rate can be captured by the cost function, whereas the constraint
functions can capture operational limitations on control effort and
prescriptions to ensure properties such as closed-loop stability or
safety.  The flexibility of this synthesis approach makes it
particularly attractive, but we should note the caveat that, in
general, the controller $u^*$ is not available in closed
form. Instead, additional work needs to be performed in order to find
the input by solving the resulting optimization
problem~\eqref{eq:optim-based-controller}.  Independently of the
computational aspects, one needs to ensure that the resulting
controller behaves properly when employed to close the loop on the
dynamical system, hence the importance of the study of the regularity
properties of optimization-based controllers.  Next we present
different examples from the systems and controls literature where such
controllers arise and motivate the importance of studying the
regularity properties of~$u^*$.

\begin{example}\longthmtitle{Control barrier and Lyapunov
    function-based control}\label{ex:safety-critical} 
  \rm{In safety-critical applications, safe controllers are often
    designed through control barrier functions
    (CBF)~\citep{ADA-SC-ME-GN-KS-PT:19}.  Let $\dot{x}=F(x,u)$ be a
    control system with $F:\real^n\times\real^m\to\real^n$ locally
    Lipschitz. Assume that the set of safe states is given by the
    $0$-superlevel set of a continuously differentiable function
    $h:\real^n\to\real$. This function is a CBF if, for every
    $x \in \real^n$, there exists $u \in \real^m$ such that
    $\nabla h(x)^T F(x,u) + \alpha(h(x)) \geq 0$, where
    $\alpha:\real\to\real$ is an extended class $\Kc$
    function\footnote{A function $\alpha:\real\to\real$ is an extended
      class-$\Kc$ function if it is strictly increasing and
      $\alpha(0)=0$.}.  Any Lipschitz controller
    $u_{\text{cbf}}:\real^n\to\real^m$ that satisfies this inequality
    at every $x \in \real^n$ renders the closed-loop system safe
    (i.e., makes the $0$-superlevel set of $h$ forward invariant).
    This inequality can be incorporated as a constraint in an
    optimization problem defining a feedback controller. For example,
    given a nominal controller $k:\real^n\to\real^m$, designed with
    desirable properties such as asymptotic stability of an
    equilibrium point or minimizing a certain infinite-horizon optimal
    control cost, \textit{safety filters}~\citep{LW-ADA-ME:17} seek to
    find the controller closest to $k$ that satisfies the CBF
    constraint. Such controller can be found at every state
    $x\in\real^n$ by solving the following optimization problem:
    \begin{subequations}\label{eq:safety-filter}
      \begin{align}
        &  u_{\text{sf}}^*(x) = \underset{u\in\real^m}{\text{argmin}} \
          \frac{1}{2}\norm{u-k(x)}^2,
        \\
        & \qquad \text{s.t.} \ \nabla h(x)^T F(x,u) + \alpha(h(x)) \geq 0,
      \end{align}  
    \end{subequations}
    where again $\alpha:\real\to\real$ is an extended class $\Kc$
    function.
    Often, one also seeks to endow safety filters with stability
    guarantees by employing control Lyapunov functions
    (CLF)~\citep{RAF-PVK:96a}.  Given a positive definite function
    $V:\real^n\to\real_{\geq0}$, $V$ is a CLF if, for all
    $x \in \real^n \setminus \{0\}$, there exists $u \in \real^m$ such
    that $\nabla V(x)^T F(x,u) + W(x) \leq 0$, where
    $W:\real^n\to\real$ is a positive definite function.  A locally
    Lipschitz controller $u_{\text{clf}}:\real^n\to\real^m$ that
    satisfies this inequality at every state $x\in\real^n$ renders the
    origin of the closed-loop system asymptotically stable.  Given a
    CLF $V:\real^n\to\real$, one can seek to endow $u_{\text{sf}}$
    with stability guarantees by solving the following optimization
    problem at every $x\in\real^n$:
    \begin{subequations}\label{eq:clf-cbf-controller}
      \begin{align}
        & u_{\text{cc}}^*(x) =  \underset{u\in\real^m}{\text{argmin}} \
          \frac{1}{2}\norm{u-k(x)}^2,
        \\
        & \qquad \text{s.t.} \ \nabla h(x)^T F(x,u) + \alpha(h(x)) \geq 0, \\
        & \qquad \qquad \ \nabla V(x)^T F(x,u) + W(x) \leq 0,
      \end{align}  
    \end{subequations}
    Note that both~\eqref{eq:safety-filter}
    and~\eqref{eq:clf-cbf-controller} are special cases
    of~\eqref{eq:optim-based-controller}.  Similar optimization-based
    control designs of the form~\eqref{eq:optim-based-controller}
    leveraging CLFs and CBFs have been proposed
    in~\citep{ADA-XX-JWG-PT:17,PM-JC:23-csl,YY-SK-BG-NA:23,FC-JJC-BZ-CJT-KS:21,PS-MJ-EH:21,KL-YY-JC-NA:23-acc,AJT-VDD-SD-BR-YY-ADA:21,PM-CNG-JC:24-ral},
    among many others.  
    If the system is control-affine, as it is often the case in practice, 
    then~\eqref{eq:safety-filter} and~\eqref{eq:clf-cbf-controller}
    are quadratic programs (QPs).
    Importantly, $u_{\text{sf}}^*$
    (resp. $u_{\text{cc}}^*$) is only guaranteed to be safe
    (resp. safe and stable) if it is locally Lipschitz. Hence,
    studying the regularity properties of~\eqref{eq:safety-filter}
    and~\eqref{eq:clf-cbf-controller} is critical to ensure that the
    closed-loop system has the desired safety and/or stability
    properties.  Moreover, if $u^*$ is locally Lipschitz, then the
    right-hand side of~\eqref{eq:closed-loop} is locally Lipschitz
    too, and then the Picard-Lindel\"of theorem~\citep[Theorem
    2.2]{GT:12-ams} guarantees existence and uniqueness of solutions
    for small enough times. Similar regularity properties are also
    relevant in the study of the contraction properties of
    optimization-based controllers of the
    form~\eqref{eq:safety-filter} and~\eqref{eq:clf-cbf-controller},
    as shown in~\citep{AD-FB:24}.
  } \demo
\end{example}

\begin{example}\longthmtitle{Online feedback
    optimization}\label{ex:feedback-opt}
  \rm{Here we describe the problem of optimally regulating the
    steady-state output of a plant, a task often referred to as
    \emph{online feedback
      optimization}~\citep{MC-EDA-AB:20,AH-SB-GH-FD:21}. This problem arises
    in a variety of application areas including power
    systems~\citep{ED-AS:16,LSPL-JWSP-EM:21}, network congestion
    control~\citep{SHL-FP-JCD:02}, and traffic
    networks~\citep{GB-JC-JIP-EDA:22-tcns}.
    In a typical set-up, the plant is modeled with the dynamics
    \begin{align}
      \label{eq:input-output-dynamics}
      \dot{x} &= F(x, u) \notag
      \\ 
      y &= G(x, u)
      \end{align}
    where $G:\real^n \times \real^m \to \real^k$ and $y \in \real^k$
    denote the output.
    We assume that there exists a map
    $h:\real^m \to \real^k$, called the \emph{steady-state map}, such
    that for each  constant input $u \in \real^m$ and initial
    condition $x_0 \in \real^n$, the corresponding output of
    \eqref{eq:input-output-dynamics} satisfies $y(t) \to h(u)$ as
    $t \to \infty$. Consider the problem of driving the output to an
    optimal steady-state, formalized by the optimization
    \begin{subequations}\label{eq:optimal-steady-state}
      \begin{align}
        & \underset{u\in\real^m, y \in \real^k}{\text{min}} \quad \
          \Phi(u, y)
        \\ 
        &\text{s.t.} \qquad \qquad \ (u, y) \in \Uc \times \Yc \\ 
        & \qquad \qquad \qquad y = h(u),
      \end{align}  
    \end{subequations}
    where~$\Yc \subset \real^k$ denotes the set of valid
    outputs,~$\Uc \subset \real^m$ denotes the set of valid inputs,
    and~$\Phi:\real^m \times \real^k \to \real$ denotes the cost of
    the corresponding input-output pair.  The problem can equivalently
    be viewed as an optimization over a set of inputs alone by
    eliminating the variable~$y$ from~\eqref{eq:optimal-steady-state}:
    \begin{subequations}\label{eq:optimal-feedforward}
      \begin{align}
        & \underset{u\in\real^m}{\text{min}} \quad \quad \ \Phi(u,
          h(u))
        \\
        &\text{s.t.} \qquad \quad (u, h(u)) \in \Uc \times \Yc.
      \end{align}  
    \end{subequations}
    Note that \eqref{eq:optimal-steady-state} and
    \eqref{eq:optimal-feedforward} are ``static'' problems, i.e., the
    state of the plant does not appear in the cost function or the
    constraints as a parameter.  In practice, however, the
    steady-state map and the plant dynamics are only partially known,
    and subject to external disturbances or model uncertainties. This
    precludes one from directly solving either problem offline and
    simply applying the resulting input
    to~\eqref{eq:input-output-dynamics} (this strategy is called
    \emph{feedforward optimization}).  Instead, one turns the
    ``static'' formulation into a ``dynamic'' one by solving the
    problem online and using real-time measurements of the output of
    the plant in place of a closed-form expression of the steady-state
    output.  Formally, this amounts to replacing the expression~$y$ in
    \eqref{eq:optimal-steady-state}, or~$h(u)$ in
    \eqref{eq:optimal-feedforward}, with~$G(x, u)$:
    \begin{subequations}\label{eq:optimal-feedback}
      \begin{align}
        u^*_\text{ofo}(x) =\
        &\underset{u\in\real^m}{\text{argmin}}\!\! &&f(x, u):=\Phi(u,
                                                      G(x, u))
        \\ 
        &\text{s.t.} && u \in \Uc
        \\
        & && G(x, u) \in \Yc.
    \end{align}  
  \end{subequations}
  The idea is that at each time instant, the output measurement is
  obtained online and fed back into~\eqref{eq:optimal-feedback}, hence
  this strategy is called \emph{online feedback optimization}.
  %
  We note that~\eqref{eq:optimal-feedback} is of the
  form~\eqref{eq:optim-based-controller}, by rewriting the input and
  output set inclusions as inequalities as
  in~\eqref{eq:optim-based-controller}.  In this setting,
  understanding the regularity properties of the closed-loop dynamics
  $\dot{x} = F(x, u^*_\text{ofo}(x))$ becomes necessary to ensure good
  performance of implementing \eqref{eq:optimal-feedback} on a
  physical plant.  Letting $u_{\text{ss}}$ denote the solution
  to~\eqref{eq:optimal-feedforward}, one would be interested, for
  instance, in showing that
  $u^{*}_{\text{ofo}}(x(t)) \to u_{\text{ss}}$ and
  $y(t) \to h(u_{\text{ss}})$ as $t \to \infty$.  We
    also note that if the cost function
    in~\eqref{eq:optimal-steady-state} is time-varying, the resulting
    ``dynamic'' formulation~\eqref{eq:optimal-feedback} is also
    time-varying. However, the added time-dependence can be treated as
    an extra parameter in the optimization
    problem~\eqref{eq:optimal-feedback}, and therefore the results
    outlined in this paper can also be applied in such time-dependent
    settings.
} \demo
\end{example}

\begin{example}\longthmtitle{Optimization algorithms as dynamical
    systems}\label{ex:optim-problems-dynamical-sys}
  \rm{Optimization algorithms can be viewed from the lens of
    dynamical systems~\citep{KA-LH-HU:58,RWB:91,UH-JBM:94}.  In some
    cases, such dynamical systems are designed using ideas from
    optimization-based control.  Here we discuss the \textit{safe
      gradient flow}~\citep{AA-JC:24-tac}. Consider a constrained
    optimization problem of the form
    \begin{subequations}
      \label{eq:opt-pb-static}
      \begin{align}
        &\min\limits_{x\in\real^n} \quad \bar{f}(x),
        \\
        &\text{s.t.} \qquad \bar{g}(x) \leq 0
    \end{align}
  \end{subequations}
  where $\bar{f}:\real^n\to\real$ and $\bar{g}:\real^n\to\real^m$ are
  continuously differentiable functions with Lipschitz continuous
  gradients, and $\frac{\partial \bar{g}}{\partial x}$ has full rank
  for all $x \in \real^n$.  We consider the problem of designing a
  continuous-time dynamical system such that the feasible
  set~$\Cc = \{x \in \real^n \mid \bar{g}(x) \leq 0 \}$ is forward
  invariant and trajectories converge to solutions
  to~\eqref{eq:opt-pb-static}.  To solve this problem, we consider the
  integrator system,
  \begin{align}
    \label{eq:sgf-dynamics}
    \dot{x} = F(x, \xi) = \xi,
  \end{align}
  along with the feedback controller
  \begin{subequations}
    \label{eq:safe-gradient-flow}
    \begin{align}
      \xi^*_\alpha(x) =
      & \underset{\xi\in\real^n}{\text{argmin}} \
        \frac{1}{2}\norm{\xi+\nabla \bar{f}(x)}^2
      \\
      & \text{s.t} \quad \frac{\partial \bar{g}(x)}{\partial x}\xi
        \leq -\alpha \bar{g}(x),
    \end{align}
  \end{subequations}
  where $\alpha > 0$ is a design parameter. The closed-loop dynamics
  are referred to as the \emph{safe gradient flow} (cf. \citep[Section
  IV.A]{AA-JC:24-tac} for an alternative derivation of the safe
  gradient flow using techniques from the theory of control barrier
  functions).  We note again that~\eqref{eq:safe-gradient-flow} is of
  the form~\eqref{eq:optim-based-controller}.  Establishing regularity
  properties of $\xi^*_{\alpha}$ such as continuity or local
  Lipschitzness is critical for the solutions
  of~\eqref{eq:safe-gradient-flow} to exist and be unique.  These
  properties are
  %
  %
  then leveraged to study the convergence of the solutions
  of~\eqref{eq:safe-gradient-flow} to the optimizers
  of~\eqref{eq:opt-pb-static} while ensuring that the feasible set is
  forward invariant.  }\demo
\end{example}

\begin{example}\longthmtitle{Projected Dynamical Systems}
  \rm{Projected dynamical systems are a class of systems whose
    evolution is constrained to remain inside a subset
    $\Cc \subset \real^n$.  They are are widely used for analyzing and
    solving nonlinear programs and variational inequalities
    \citep{AN-DZ:96} and have wide-ranging applications including 
    network economics (e.g., for analyzing supply chain networks or
    financial markets) \citep{AN:09}, power networks
    \citep{ED-AS:16,DKM-FD-SH-SHL-SC-BR-JL:17}, anti-windup
    controllers for feedback optimization \citep{AH-FD-AT:20}, and
    traffic flows \citep{AN-DZ:97}, to name a few.  While
    projected dynamical systems have been considered in quite general
    settings, such as on Riemmanian manifolds \citep{AH-SB-FD:21}, or
    with respect to oblique projections \citep{WPMHH-MKC-MFH:20}, here
    we restrict ourselves to the Euclidean case. In this case,
    projected dynamical systems typically take the form
    \begin{equation}
      \label{eq:PDS}
      \dot{x} = \Pi_\Cc[x, H(x)]
    \end{equation}
    where $H:\real^n \to \real^n$ is a vector field, and
    $v \mapsto \Pi_\Cc[x, v]$ is the projection onto the tangent cone
    of~$\Cc$:
    \[
      \Pi_\Cc[x, v] = \proj_{T_\Cc(x)}(v).
    \]
    Recently, projected dynamical systems have been reinterpreted from
    the viewpoint of control theory, to design anytime flows solving
    variational inequalities \citep{AA-JC:23-auto}, and for
    understanding their relationship to controllers obtained using
    techniques from the theory of control barrier functions
    \citep{GD-JC-WPMHH:24}. For example,
    \eqref{eq:PDS} can be interpreted as the closed-loop dynamics
    corresponding to the system,
    \begin{equation}
      \label{eq:PDS-open-loop}
      \dot{x} = H(x) + u 
    \end{equation}
    with the feedback controller
    \begin{subequations}\label{eq:PD-feedback}
      \begin{align}
        u_\text{proj}^*(x) =\
        &\underset{u\in\real^m}{\text{argmin}}\!\! &&\norm{u}^2
        \\ 
        &\text{s.t.} && H(x) + u \in T_\Cc(x).
      \end{align}  
    \end{subequations}
    Even though feedback controllers of the form
    \eqref{eq:PD-feedback} are discontinuous, the resulting
    closed-loop system may still be well behaved.  In the case where
    $\Cc$ is convex, one can show existence and
    uniqueness~\citep[Theorem 2.5]{AN-DZ:96} of Carath\'eodory
    solutions (cf. \citep{JC:08-csm} for notions of solutions to
    discontinuous systems), and forward invariance of the set $\Cc$
    \citep[Corollary 4.8]{FB-SM:08}. With the additional assumption of
    strong monotonicity of $H$, asymptotic stability of the unique
    equilibrium~\eqref{eq:PDS} follows as well.  The control-theoretic
    interpretation of projected dynamical systems highlights the
    complex relationship between the regularity properties of
    optimization-based feedback controllers and the dynamical
    properties of the resulting closed-loop system. In particular, it
    shows that from the perspective of control design, a feedback
    controller may achieve its intended objective (e.g., ensuring
    invariance of a safe set or stabilization to a desired equilibrium
    point) with relatively weak regularity properties.
  }
  %
  %
  \demo
\end{example}

\begin{example}\longthmtitle{Model predictive control}\label{ex:mpc}
  \rm{Here we explain how~\eqref{eq:optim-based-controller} is also
    applicable to model predictive controllers. Consider a
    discrete-time dynamical system
    \begin{align}\label{eq:discrete-time}
      x^+ = F(x, u),
    \end{align}
    where $x \in \real^n$ and $u \in \real^m$.  We consider the
    problem of optimally controlling \eqref{eq:discrete-time} to
    minimize a running cost $\ell(x, u)$ while ensuring the state and
    input satisfy constraints $x \in \Xc \subset \real^n$ and
    $u \in \Uc \subset \real^m$.  Model predictive control is a
    method for solving this problem by solving a finite-horizon
    optimal control problem and implementing its solution over
    \eqref{eq:discrete-time} in a receding horizon fashion. Here we
    show that model predictive control schemes can be interpreted as a
    discrete-time analog of optimization-based feedback control
    discussed in previous examples.  Let $N > 0$ be a time horizon,
    and
    $\textbf{x} = (\hat{x}_0, \hat{x}_1, \dots, \hat{x}_{N}) \in
    \real^{n(N + 1)}$ and
    $\textbf{u} = (\hat{u}_0, \hat{u}_1, \dots, \hat{u}_{N - 1}) \in
    \real^{mN}$ denote the state and input prediction sequences over
    the time horizon.  Consider the following optimization problem
    \begin{subequations}\label{eq:mpc-controller}
      \begin{align}
        \textbf{u}_{\text{mpc}}^*(x) =\
        &\underset{\textbf{u}, \textbf{x}}{\text{argmin}} \
        &&\!\!\!\!\!\!\!V_N(\hat{x}_N) +
           \sum_{k=0}^{N-1}\ell(\hat{x}_k, \hat{u}_k)
        \\
        &\text{s.t.} &&\hat{x}_{k+1} = F(\hat{x}_{k}, \hat{u}_{k})
        \\
        & && \hat{x}_k \in \Xc,\ \hat{u}_k \in \Uc
        \\
        & && \hat{x}_N \in \Xc_f
        \\
        & && \hat{x}_0 = x
        \\
        & && k\in\{0,\hdots,N-1 \},
      \end{align}  
    \end{subequations}
    where $V_N:\real^n \to \real$ and $\Xc_f \subset \real^n$ denote
    the terminal cost and terminal constraint respectively (we refer
    the reader to~\citep{JBR-DQM-MMD:17} for conditions on these
    ingredients to ensure closed-loop stability.)  Next, consider the
    augmented system
    \begin{subequations}
      \label{eq:dt-augmented}
      \begin{align}
        x^+ &= \bar{F}(x, \textbf{u})
        \\
        &= F(x, \hat{u}_0),
      \end{align}
    \end{subequations}
    which simply corresponds to implementing the first input in the
    sequence $\textbf{u}$ to~\eqref{eq:discrete-time}. Note that
    \eqref{eq:mpc-controller} is a parametric optimization problem,
    where the parameter corresponds to the current
    state~$x$. Model predictive control studies 
    the closed-loop system obtained by the dynamics
    \eqref{eq:dt-augmented} and the controller
    \eqref{eq:mpc-controller}. As shown in~\citep{POMS-JBR-ESM:97},
    establishing the continuity properties
    of~\eqref{eq:mpc-controller} is critical in proving stability and
    robustness properties of MPC-based controllers.
}\demo
\end{example}

As motivated by the examples provided above, studying the regularity
properties of $u^*$ is critical in order to establish different
properties of interest for the closed-loop system, such as
\begin{enumerate}
\item existence and uniqueness of solutions (for different notions of
  solution, such as classical, Carathéodory, and Filippov);
\item dynamical properties such as forward invariance of safe sets or
  stabilization to an equilibrium point;
\item convergence of optimization algorithms such as the safe gradient
  flow;
\item good performance of online feedback optimization-based
  controllers;
\item stability and robustness properties of MPC based controllers.
\end{enumerate}
Additionally, from a practical point of view, guaranteeing regularity
properties for $u^*$ such as continuity or local Lipschitzness is
useful to ease the implementation of such controllers on digital
platforms and avoid chattering behavior.

\subsection{State of the Art}
Having established the importance of characterizing the regularity
properties of~\eqref{eq:optim-based-controller}, we next discuss the
state of the art.  There are a variety of works in the
literature~\citep{BJM-MJP-ADA:15,AI-MG-RGS-WED:22,PM-JC:24-auto,PM-KL-NA-JC:24-csl}
and~\citep[Theorem 2.7]{JBR-DQM-MMD:17} that use the theory of
parametric optimization to guarantee local Lipschitz continuity or
other regularity properties of optimization-based controllers.  For
example, the results in~\citep{BJM-MJP-ADA:15} give different
conditions that ensure continuity and continuous differentiability of
optimization-based controllers.  However, they either require the
rather strong assumption of \textit{strict complementary slackness},
which is not satisfied in many cases of interest, or are limited to
quadratic programs that satisfy a set of technical conditions. The
paper~\citep{BJM-MJP-ADA:15} also revisits Robinson's counterexample,
first introduced in~\citep{SMR:82}, in the context of
optimization-based control, which shows that even for optimization
problems defined by well-behaved data (e.g., second-order continuously
differentiable
objective function and constraints, strongly convex objective
function, and feasible set with non-empty interior, which are widely
employed in the design of safe and stabilizing controllers,
cf. Example~\ref{ex:safety-critical}), the resulting controller might
not be locally Lipschitz.  The result in~\citep[Theorem
3]{AI-MG-RGS-WED:22} is more general but only ensures continuity under
\textit{Slater's condition} and other regularity properties on the
optimization problem. The regularity results in our previous
work~\citep{PM-JC:24-auto,PM-KL-NA-JC:24-csl} establish different
Lipschitz continuity results for second-order convex programs, but are
limited to this specific type of optimization problems.
Finally,~\citep[Theorem 2.7]{JBR-DQM-MMD:17} only guarantees
continuity of optimization-based controllers derived from MPC. We also
note that in some cases, $u^*$ can be computed in closed-form, in
which case the regularity properties of $u^*$ can be evaluated
directly without having to resort to the theory of parametric
optimization. Examples of such explicitly computable
optimization-based controllers are provided
in~\citep{XT-DVD:24,PM-JC:23-csl,YY-SK-BG-NA:23,MA-NA-JC:23-tac}
and~\citep[Chapter 7]{JBR-DQM-MMD:17}.  We would also like to point
out that even though this work is mostly focused on control laws
obtained as the solution of optimization problems of the
form~\eqref{eq:optim-based-controller}, the regularity properties of
other control designs has also been studied in the literature.  For
example, the celebrated Sontag's Universal Formula~\citep{EDS:89}
provides a smooth control law for stabilization of control-affine
systems. More recently, similar designs have been given in the context
of safety-critical
control~\citep{MC-PO-GB-ADA:23,PM-JC:24-auto,ML-ZS-SW:24,PO-JC:19-cdc}.

\begin{table*}[htbp]
  \centering
  \small
  \begin{tabular}{ |c | c | c | c |}
    \hline
    Assumptions & Regularity of $u^*$ & Existence & Uniqueness \\
    \hline
    \begin{tabular}{@{}c@{}} $f, g$ analytic \\ $f(x,\cdot)$ strictly convex $\forall x\in\real^n$ \\ $g(x,\cdot)$ convex $\forall x\in\real^n$ \\ Existence of minimizer \\ LICQ and SCS \end{tabular}
    & \begin{tabular}{@{}c@{}} analytic \\ cf.~\citep{AVF:76} \end{tabular} & \cmark & \cmark \\
    \hline
    \begin{tabular}{@{}c@{}} $f, g\in\Cc^{p}(\real^n\times\real^m)$ \\ $p\in\mathbb{Z}_{>0}, p\geq2$ \\ $f(x,\cdot)$ strictly convex \\ $g(x,\cdot)$ convex \\ Existence of minimizer \\ LICQ and SCS \end{tabular}
    & \begin{tabular}{@{}c@{}} $\Cc^{p-1}$ \\ cf.~\citep{AVF:76} \end{tabular} & \cmark & \cmark \\
    \hline
    \begin{tabular}{@{}c@{}} $f, g\in\Cc^{2}(\real^n\times\real^m)$ \\ $f(x,\cdot)$ strictly convex $\forall x\in\real^n$ \\ $g(x,\cdot)$ convex $\forall x\in\real^n$ \\ Existence of minimizer, LICQ \end{tabular}
    & \begin{tabular}{@{}c@{}} Locally Lipschitz \\ cf.~\citep{SMR:80} \end{tabular} & \cmark & \cmark \\
    \hline
    \begin{tabular}{@{}c@{}} $f, g\in\Cc^{2}(\real^n\times\real^m)$ \\ $f(x,\cdot)$ strictly convex $\forall x\in\real^n$ \\ $g(x,\cdot)$ convex $\forall x\in\real^n$ \\ Existence of minimizer \\ CR and MFCQ \end{tabular}
    & \begin{tabular}{@{}c@{}} Locally Lipschitz \\ cf.~\citep{JL:95} \end{tabular} & \cmark & \cmark \\
    \hline
    \begin{tabular}{@{}c@{}} $f, g\in\Cc^{2}(\real^n\times\real^m)$ \\ $f(x,\cdot)$ strongly convex $\forall x\in\real^n$ \\ $g(x,\cdot)$ convex $\forall x\in\real^n$ \\ Existence of minimizer, SC \end{tabular}
    & \begin{tabular}{@{}c@{}} Point-Lipschitz and H\"older,  \\ cf. Proposition~\ref{prop:regularity}, and \\ {locally Lipschitz for scalar} \\ QPs, cf. Proposition~\ref{prop:scalar-qp-lipschitz} \end{tabular} & \cmark & \begin{tabular}{@{}c@{}} Only in special cases \\ cf. Proposition~\ref{prop:point-lipschitzness-uniqueness} \\  Corollary~\ref{cor:pl-1d-uniqueness}, Example~\ref{ex:point-lip-nonunique} \end{tabular} \\
    \hline
    \begin{tabular}{@{}c@{}} $f, g\in\Cc^{2}(\real^n\times\real^m)$ \\ $f(x,\cdot)$ strictly convex $\forall x\in\real^n$ \\ $g(x,\cdot)$ convex $\forall x\in\real^n$ \\ Existence of minimizer, SC \end{tabular}
    & \begin{tabular}{@{}c@{}} Directionally differentiable,  \\ cf. Proposition~\ref{prop:regularity}, \\ {locally Lipschitz for scalar} \\ QPs, cf. Proposition~\ref{prop:scalar-qp-lipschitz}, and \\ continuous, cf.~\citep[Thm 5.3]{AVF-JK:85}, \\  but might not be \\ point-Lipschitz, \\ cf. Example~\ref{ex:no-PL-without-diff-wrt-parameter} \end{tabular} & \cmark & \begin{tabular}{@{}c@{}} \xmark \\ cf. Example~\ref{ex:point-lip-nonunique}  \end{tabular} \\
    \hline
    \begin{tabular}{@{}c@{}} $f, g\in\Cc^{0}(\real^n\times\real^m)$ \\ $f(x,\cdot)$ strictly convex $\forall x\in\real^n$ \\ $g(x,\cdot)$ convex $\forall x\in\real^n$ \\ Existence of minimizer \\ LCF $\forall x\in\real^n$ \end{tabular}
    & \begin{tabular}{@{}c@{}} Locally bounded \\ {cf. Proposition 3.7, and} \\ measurable \\ {cf. Proposition 3.8} \end{tabular}
    & \begin{tabular}{@{}c@{}} \xmark \ (classical) \\ \cmark \ (Filippov) \end{tabular} 
    & \begin{tabular}{@{}c@{}} \xmark \ (classical) \\ \xmark \ (Filippov) \end{tabular} \\
    \hline
    \begin{tabular}{@{}c@{}} $f, g\in\Cc^{2}(\real^n\times\real^m)$ \\ $f(x,\cdot)$ strongly convex $\forall x\in\real^n$ \\ $g(x,\cdot)$ convex $\forall x\in\real^n$ \\ Existence of minimizer \end{tabular}
    & \begin{tabular}{@{}c@{}} Might be discontinuous \\ and even unbounded \\ cf. Examples~\ref{ex:discontinuous-no-SC},~\ref{ex:unbounded-no-SC} \end{tabular}
    & \begin{tabular}{@{}c@{}} \xmark \ (classical) \\ \xmark \ (Filippov) \end{tabular}
    & \begin{tabular}{@{}c@{}} \xmark \ (classical) \\ \xmark \ (Filippov) \end{tabular} \\
    \hline
  \end{tabular}
  \caption{Summary of results on regularity properties of
    optimization-based controllers. The first column describes the
    different assumptions.  The second column describes the regularity
    properties of $u^*$. The third (resp.  fourth) column describes
    whether existence (resp. uniqueness) of classical solutions of the
    closed-loop system~\eqref{eq:closed-loop} is guaranteed (provided
    that $F:\real^n\times\real^m\to\real^n$ is locally Lipschitz). In
    the last two columns, properties are stated by default for
    classical solutions. If results are available for both classical
    and Filippov solutions, the property for each type of solution is
    denoted separately.  LICQ stands for \textit{linear independence
      constraint qualification}, SCS stands for \textit{strict
      complementary slackness}, MFCQ stands for
    \textit{Mangasarian-Fromovitz Constraint Qualification} and CR
    stands for \textit{constant rank condition}. The terminology for
    regularity and constraint qualification is given in
    Section~\ref{sec:preliminaries}.}
  \label{tab:1}
\end{table*}

\normalsize

\subsection{Paper Goals and Contributions}

Our main goal in writing this paper is to provide an integrative
presentation of insights and results about the regularity of
optimization-based controllers.  We present in Table~\ref{tab:1} a
comprehensive collection of results that offers the reader interested
in using optimization-based controllers a one-stop destination to
assess the regularity properties of their control design.  The paper
presents several results from the literature, but restated here for
completeness from the perspective of optimization-based control. The
paper also contains many novel results that help fill gaps in the
state of the art.

In what follows, we assume that the control system operates in
continuous time and is given by
\begin{equation}\label{eq:dynamics}
  \dot{x} = F(x, u),
\end{equation}
where $F:\real^n\times\real^m\to\real^n$ is locally Lipschitz. Hence,
the closed-loop system takes the form
\begin{align}\label{eq:closed-loop}
  \dot{x}=F(x,u^*(x)).
\end{align}

On the technical level, the contributions of the paper are as
follows. First, we show that under appropriate constraint
qualifications and regularity properties of the optimization
problem~\eqref{eq:optim-based-controller}, the resulting optimization-based
controller is continuous, locally Lipschitz, continuously differentiable,
and even analytic. We provide specific conditions for each of these
cases and observe that any of those conditions guarantee existence and
uniqueness of solutions for the closed-loop system. Second, given the
importance of Robinson's counterexample in showing that
optimization-based controllers defined by well-behaved data might not
be locally Lipschitz and its implications, e.g., for safety-critical
control, cf. Example~\ref{ex:safety-critical} (where most
optimization-based controllers are defined by problem data sharing the
properties of Robinson's counterexample), we characterize the
regularity properties enjoyed by the parametric optimizer of problems
defined by objective and constraints with the same properties as in
Robinson's counterexample.
We show that even though such parametric optimizers are not locally
Lipschitz in general, they enjoy other weaker regularity properties,
which are enough to guarantee existence of solutions for the
closed-loop system and in some special cases, even uniqueness.  Third,
we provide different examples that show how if the properties in
Robinson's counterexample do not hold, optimization-based controllers
can be discontinuous and in some cases, even unbounded.  Fourth, we
show that even if the optimization-based controller is discontinuous,
under appropriate regularity properties of the optimization
problem~\eqref{eq:optim-based-controller}, the parametric optimizer is
measurable and locally bounded, and the closed-loop system has
Filippov solutions.  Finally, given a safe set of interest, we study
under what regularity conditions on~\eqref{eq:optim-based-controller}
and the set, solutions of the closed-loop system remain in the safe
set, both for classical and Filippov solutions.

\section{Preliminaries}\label{sec:preliminaries}
In this section we discuss different preliminary results on regularity
of functions and constraint qualifications.

\subsection{Notions of regularity of functions}\label{sec:regularity-subsec}

Throughout the note, we make use of the following notions of
regularity of functions.

\begin{definition}\longthmtitle{Notions of Lipschitz continuity}
  A function $f:\real^n\to\real^q$ is
  \begin{itemize}
  \item \emph{point-Lipschitz at $x_0 \in \real^n$} if there exists a
    neighborhood $\Uc$ of $x_0$ and a constant $L\geq0$
    such that
  \begin{align}
    \norm{f(x)-f(x_0)} \leq L\norm{x-x_0}, \quad \forall x \in \Uc.
  \end{align}
\item \emph{locally Lipschitz at $x_0 \in \real^n$} if there exists a
  neighborhood $\tilde{\Uc}$ of $x_0$ and a constant
  $\tilde{L}\geq0$ such that
  \begin{align}
    \norm{f(x)-f(y)} \leq \tilde{L}\norm{x-y}, \quad \forall x,y \in
    \tilde{\Uc}. 
  \end{align}
\end{itemize}
\end{definition}

The notion of point-Lipschitz continuity is used, for instance,
in~\citep[Section 6.3]{GS:18} and called \textit{Lipschitz stability},
without clearly acknowledging the difference with the notion of local
Lipschitz continuity. In the optimization literature, this
  property is sometimes referred to as \emph{calmness}
  (cf. \citep[Chapter 8.F]{RTR-RJBW:98}).
%
%
Studying point-Lipschitz
continuity is natural in the context of parametric optimization, as
one is normally interested in understanding the changes in the
solution with respect to a \emph{fixed} value of the parameter.
Locally Lipschitz functions are point-Lipschitz, but the converse is
not true.  For instance, the function $f:\real\to\real$ defined by
$f(x)=x\sin(\frac{1}{x})$ if $x\neq0$ and $f(0)=0$ is point-Lipschitz
but not locally Lipschitz at the origin.

\begin{definition}\longthmtitle{H\"older property}
  A function $f:\real^n\to\real^q$ has the \emph{H\"older property at
    $x_0 \in \real^n$} if there exists a neighborhood $\hat{\Uc}$ of
  $x_0$ and constants $C>0$, $\alpha\in (0,1]$ such that
  \begin{align}
    \norm{f(x)-f(y)} \leq C \norm{x-y}^{\alpha}, \quad \forall x, y
    \in \hat{\Uc}. 
  \end{align}
\end{definition}
\smallskip

Note that if $f$ is locally Lipschitz at $x_0$ then it also has the
H\"older property at $x_0$ but the converse is not true.

\begin{definition}\longthmtitle{Directionally differentiable function}
  A function $f:\real^n\to\real$ is directionally differentiable if
  for any vector $v\in\real^n$, the limit
  \begin{align*}
    \lim_{h\to0}\frac{f(x+hv)-f(x)}{h}
  \end{align*}
  exists.  A vector-valued function
  is directionally differentiable if each of its components is
  directionally differentiable.
\end{definition}

Let $\Omega \subset \real^n$.  Throughout the paper, a function
$\varphi:\Omega \to \real^d$ belongs to the set $\Cc^{k}(\Omega)$ if
$\varphi$ is $k$-times continuously differentiable in $\Omega$. A function
$\varphi:\Omega \to \real^d$ belongs to the set $\Cc^{0}(\Omega)$ if
$\varphi$ is continuous in $\Omega$. In case we view
the elements in $\Omega$ as vectors of the Cartesian product
$\real^{m_1} \times \real^{m_2}$ and $\varphi$ takes the form
$(x, u) \mapsto \varphi(x, u)$, the function
$\varphi \in \Cc^{0, k}(\Omega)$ if for every $x\in\Omega$, 
the derivatives of order up to $k$ of $\varphi(x,\cdot)$ with respect to 
$u$ exist and are continuous with respect to $x$ and $u$.

\begin{definition}\longthmtitle{Analytic function}\label{def:analytic}
  A function $f:\real^n\to\real$ is analytic in an open set $D$ if for
  any $x\in D$ there exists a sequence
  $\{ a_n \}_{n\in\mathbb{Z}_{\geq 0}}$ such that
  $f(x)=\sum_{n=0}^{\infty} a_n (x-x_0)^n$ for all $x$ in a
  neighborhood of $x_0$.  A vector-valued function is analytic in an
  open set $D$ if each of its components is analytic.
\end{definition}

Note that an analytic function in an open set~$D$ belongs to
$\Cc^k(D)$ for any $k\in\mathbb{Z}_{\geq0}$.  Finally, we introduce
the last notion of regularity, which is weaker than all the ones
presented above and only requires the function to be bounded in a
neighborhood of a point.

\begin{definition}\longthmtitle{Locally bounded function}\label{def:loc-bounded}
  A function $f:\real^n\to\real^q$ is locally bounded at
  $x_0\in\real^n$ if there exists a neighborhood $\breve{\Uc}$ of
  $x_0$ and a constant $B>0$ such that $\norm{f(x)}\leq B$ for all
  $x\in\breve{\Uc}$.
\end{definition}


\subsection{Constraint Qualifications and
  Conditions}\label{sec:constraint-qualifications}

Here we recall different constraint qualifications and conditions for
problem~\eqref{eq:optim-based-controller}
following~\citep{SB-LV:04,GS:18}.  Throughout this section we fix
$x\in\real^n$. Furthermore, given $u\in\real^m$, we let $\Ic(x,u)$ be
the set of active constraints at $(x,u)$, i.e.,
$\Ic(x,u):=\setdef{i\in\{1,\hdots p\} }{g_i(x,u) = 0}$. We consider
the following:
\begin{description}
\item[MFCQ:] Mangasarian-Fromovitz Constraint Qualification (MFCQ)
  holds at $(x,u)\in\real^n\times\real^m$ if there exists
  $z\in\real^m$ such that $\nabla_u g_i(x,u) z < 0$ for all
  $i\in\Ic(x,u)$;
\item[LICQ:] Linear Independence Constraint Qualification (LICQ) holds
  at $(x,u)\in\real^n\times\real^m$ if the vectors
  $\{ \nabla_u g_i(x,u), i\in\Ic(x,u) \}$ are linearly independent;
\item[CR:] Constant Rank condition (CR) holds at
  $(x,u)\in\real^n\times\real^m$ if for any subset
  $L \subset \Ic(x,u)$ of active constraints, there exists a
  neighborhood $\Nc$ of $(x,u)$ such that the family
  $\{ \nabla_u g_i(x,u), i\in L \}$ remains of constant rank in $\Nc$;
\item[SC:] Slater's Condition (SC) holds at $x\in\real^n$ if there
  exists $\hat{u}\in\real^m$ such that $g_i(x,\hat{u}) < 0$ for all
  $i\in\{ 1,\hdots, p \}$;
\item[SCS:] given $x\in\real^n$, let $(u^*(x),\lambda^*(x))$ be a KKT
  point for the optimization problem
  in~\eqref{eq:optim-based-controller}.  Then, $(u^*(x),\lambda^*(x))$
  satisfies Strict Complementary Slackness (SCS) if there does not
  exist $i\in\{ 1,\hdots,p \}$ such that $\lambda_i^*(x)=0$ and
  $g_i(x,u^*(x))=0$;
\item[LCF:] local compact feasibility (LCF) holds at $x\in\real^n$ if
  there exists a compact set $K\subset\real^m$ and $\delta>0$ such
  that for all $y\in\real^n$ such that $\norm{y-x}<\delta$, there
  exists $u\in K$ such that $g(y,u)\leq 0$.
\end{description}

\section{Regularity of Parametric Optimizers}\label{sec:regularity-properties}

In this section we discuss how the assumptions on the functions $f$
and $g$ defining~\eqref{eq:optim-based-controller} affect the
regularity properties of the resulting controller~$u^*$.  Throughout
this section, we assume that $f$ and $g$ belong to
$\Cc^{2}(\real^n\times\real^m)$, $f(x,\cdot)$ is
  strictly convex for all $x\in\real^n$ (for some
  specific results, we further assume that $f(x,\cdot)$ is strongly
  convex for all $x\in\real^n$, but we make it explicit if this is the
  case) and $g(x,\cdot)$ is convex for all $x\in\real^n$.
We further assume that, for each
  $x\in\real^n$,~\eqref{eq:optim-based-controller} has at least one
  minimizer (note that if $f(x,\cdot)$ is strongly convex for all
  $x\in\real^n$, this holds if~\eqref{eq:optim-based-controller} is
  feasible for all $x\in\real^n$).  By the convexity assumptions
  described above, this implies that $u^*(x)$ is a singleton for every
  $x\in\real^n$. Furthermore, the assumptions on
  convexity also ensure that for each $x\in\real^n$,
  \textit{interior-point algorithms}~\citep{YN-AN:94} can be used to
  solve~\eqref{eq:optim-based-controller}, which run in polynomial
  time when the optimization problem is a linear program, a quadratic
  program, a second-order convex program, or a semidefinite program.
%
%
This type of assumptions are very common, for instance, in CBF-based
QPs, (cf. Example~\ref{ex:safety-critical}), or in model predictive
controllers for linear systems (cf. Example~\ref{ex:mpc}),
for which there also exist specific algorithms that
solve them efficiently~\citep{BS-GB-PG-AB-SB:20}.

First, we gather a few existing results from the literature:
\begin{description}
\item[ Continuity: ] Under the assumption
    that MFCQ holds at $(x,u^*(x))$, the parametric solution $u^*$ is
    continuous~\citep[Theorem 5.3]{AVF-JK:85}.
\item[Local Lipschitzness:] Under the assumption that both MFCQ and CR
  hold at $(x,u^*(x))$, the parametric solution $u^*$ is
locally Lipschitz~\citep[Theorem 3.6]{JL:95}.  The same conclusion can
be obtained if LICQ holds~\citep[Theorem 4.1]{SMR:80} at $(x,u^*(x))$.
We note also that since the satisfaction of LICQ implies the 
satisfaction of MFCQ and CR (cf.~\citep[Proposition 3.1]{JL:95}),~\citep[Theorem 3.6]{JL:95}  
is stronger than~\citep[Theorem 4.1]{SMR:80}.
\item[Continuous Differentiability:] Under the assumptions of LICQ and
  SCS, the parametric solution $u^*$ is continuously
  differentiable~\citep[Theorem 2.1]{AVF:76}. This last point was
  already noted in the optimization-based control literature
  in~\citep[Theorem 1]{BJM-MJP-ADA:15}. In fact, if $f$ and $g$
  belong to $\Cc^{p}(\real^n,\real^m)$, with
  $p\in\mathbb{Z}_{>0}, p\geq2$, the proof of~\citep[Theorem
  2.1]{AVF:76} can be adapted using the Implicit Function Theorem for
  higher degree of differentiability~\citep[Proposition
  1B.5]{ALD-RTR:14}, to show that the parametric optimizer belongs to
  $\Cc^{p-1}(\real^n)$.
\item[Analyticity:] Similarly, if $f$ and $g$ are analytic in
  $\real^n$, then the proof of~\citep[Theorem 2.1]{AVF:76} can be
  adapted using the Analytic Function Theorem~\citep[Theorem
  3.3.2]{MSB:77},
  to show that the parametric optimizer is analytic in $\real^n$.
\end{description}

If the constraint qualifications given above for the case of 
local Lipschitzness do not hold, the
parametric optimizer can fail to be locally Lipschitz. To illustrate
this, we revisit next an example due to Robinson~\citep{SMR:82}.

\begin{example}\longthmtitle{Robinson's
    Counterexample}\label{ex:robinson-counterex}
  \rm{In~\citep{SMR:82}, Robinson introduces the following parametric
  optimization problem: for $x=(x_1, x_2)\in\real^2$, consider
  \begin{subequations}\label{eq:robinsons-counterexample}
    \begin{align}
      & \min_{u\in\real^4} \frac{1}{2}u^\top  u
      \\
      & 
        \quad \text{s.t. } A(x)u \geq b(x)
    \end{align}  
  \end{subequations}
  where
  \begin{align*}
    A(x) = \begin{bmatrix}
            0 & -1 & 1 & 0, \\
            0 & 1 & 1 & 0, \\
            -1 & 0 & 1 & 0, \\
            1 & 0 & 1 & x_1
          \end{bmatrix} ,
      \quad b(x) = \begin{bmatrix}
          1 \\
          1 \\
          1 \\
          1+x_2
      \end{bmatrix} .
  \end{align*}
  Problem~\eqref{eq:robinsons-counterexample} is a quadratic program
  with strongly convex objective function, smooth objective function
  and constraints, and for which Slater's condition holds for every
  value of the parameter (this can be shown by noting that
  $\hat{u}=(0,0,2+|x_2|,0)$ satisfies all constraints strictly).
  Despite these nice properties, the parametric solution
  of~\eqref{eq:robinsons-counterexample} is not locally Lipschitz at
  $(x_1,x_2)=(0,0)$.  Indeed, let $u^*:\real^2\to\real^4$ denote the
  parametric solution of~\eqref{eq:robinsons-counterexample}, and
  $u_4^*:\real^2\to\real$ its fourth component, which is given by
  \begin{align*}
    u_4^*(x_1,x_2)=
    \begin{cases}
      0  &\text{if} \ x_2\leq 0,
      \\
      \frac{x_2}{x_1} \
        &\text{if} \ x_2\geq0, \ x_1\neq0,
          \frac{x_1^2}{2}\geq x_2,
      \\ 
      \frac{x_1(x_2+1)}{x_1^2+2} \ &\text{otherwise}.
    \end{cases}
  \end{align*}
  $u_4^*$ in Figure~\ref{fig:sim}.
  For $s \geq 0$, let $p(s) = (s, s^2/2)$ and $q(s) = (s, 0)$. Then, 
  observe that $p(s)$ and $q(s)$ approach the origin as $s \to 0^+$, 
  however,
  \begin{align*}
    \frac{\norm{u_4^*(p(s))-u_4^*(q(s))}}{\norm{p(s) - q(s)}}
    = \frac{1}{s}.
  \end{align*}
  Since the right hand side of the previous expression can be 
  made arbitrarily large by choose $s$ sufficiently small, it 
  follows that 
  $u^*$ is not locally Lipschitz at the origin.  We also note
  that~\citep[Example 3.11]{JL:95} gives a similar example for a
  parametric quadratic program with a two-dimensional optimization
  variable, three-dimensional parameter, strongly convex objective
  function, smooth objective function and constraints, and for which
  Slater's condition holds for every value of the parameter and 
  the corresponding parametric optimizer also fails to be locally Lipschitz.
} \demo
\end{example}

Even though Example~\ref{ex:robinson-counterex} shows that the
parametric optimizer of~\eqref{eq:robinsons-counterexample} is not
locally Lipschitz, it actually satisfies a set of weaker regularity
properties.  The following result characterizes them, in the general
setting of optimization problems satisfying the same conditions
as~\eqref{eq:robinsons-counterexample}.

\begin{proposition}\longthmtitle{Regularity Properties of Parametric
    Optimizer}\label{prop:regularity}
  Suppose that $f$ and $g$ belong to
  $\Cc^{2}( \real^n\times\real^m )$. Further assume that for any
  $x\in\real^n$, $g(x,\cdot)$ is
  convex. Suppose that SC holds at $x_0\in\real^n$.
  Then,
  \begin{enumerate}
  \item\label{it:first} if $f(x,\cdot)$ is strongly convex for all $x\in\real^n$,
    there exists a neighborhood $\Vc_{x_0}$ of
    $x_0$ such that $u^*$ is point-Lipschitz at $y$ for all
    $y\in\tilde{\Vc}_{x_0}$;
  \item\label{it:second} if $f(x,\cdot)$ is strongly convex for all $x\in\real^n$, 
   $u^*$ has the H\"older property at $x_0$;
  \item\label{it:third} if $f(x,\cdot)$ is strictly convex for all $x\in\real^n$
  and~\eqref{eq:optim-based-controller} has at least one minimizer for all $x\in\real^n$,
   $u^*$ is directionally differentiable at $x_0$.
  \end{enumerate}
\end{proposition}
\begin{proof}
  First we note that since in~\ref{it:first} and~\ref{it:second}, 
  $f(x_0,\cdot)$ is strongly convex, and in~\ref{it:third} $f(x_0,\cdot)$ is 
  strictly convex and~\eqref{eq:optim-based-controller} has at least one minimizer at 
  $x_0$, it follows that in all of~\ref{it:first},~\ref{it:second},~\ref{it:third},
  $u^*(x_0)$ is unique and well-defined for all $x_0\in\real^n$.
  
  To prove~\ref{it:first} we use~\citep[Theorem 6.4]{GS:18}.  Since SC
  holds at $x_0$, by~\citep[Prop. 5.39]{NA-AE-MP:20}, since
  $g(x_0,\cdot)$ is convex, MFCQ holds at $(x_0,u^*(x_0))$.
  Furthermore, since $f(x_0,\cdot)$ is strongly convex and
  $g(x_0,\cdot)$ is convex, the second-order condition
  SOC2~\citep[Definition 6.1]{GS:18} holds (note that SOC2 is not guaranteed 
  to hold if $f(x_0,\cdot)$ is only strictly convex). All of this, together with
  the twice continuous differentiability of $f$ and $g$ imply,
  by~\citep[Theorem 6.4]{GS:18}, that $u^*$ is point-Lipschitz at
  $x_0$. Now, since $g$ is continuous, there exists a neighborhood
  $\Vc_{x_0}$ of $x_0$ such that SC holds for all $y\in\Vc_{x_0}$.  By
  repeating the same argument, $u^*$ is point-Lipschitz at $y$ for all
  $y\in \Vc_{x_0}$.
  
  Now let us prove~\ref{it:second}. We use~\citep[Theorem
  2.1]{NDY:95}, which gives a sufficient condition for the solution of
  a variational inequality to have the H\"older property.  Recall that
  given a map $F:\real^m\to\real^m$, and a constraint set
  $\bar{\Cc}\subset\real^m$, a variational inequality refers to the
  problem of finding $u^*\in\bar{\Cc}$ such that
  $(u-u^*)^T F(u^*) \geq 0$ for all $u\in\bar{\Cc}$. For every fixed
  $x\in\real^n$, by taking the map $F$ to be the gradient of $f$ with
  respect to $u$ at $x$, and by taking $\bar{\Cc}$ to be the
  constraint set of~\eqref{eq:optim-based-controller} at $x$, a
  constrained optimization problem of the
  form~\eqref{eq:optim-based-controller} can be posed as a variational
  inequality, cf.~\citep{DK-GS:80}.  Since $f$ is twice continuously
  differentiable and $f(x_0,\cdot)$ is strongly convex, conditions (2.1) and (2.2)
  in~\citep[Theorem 2.1]{NDY:95} hold.
  Note that condition (2.2) in~\citep[Theorem 2.1]{NDY:95} is not guaranteed 
  to hold if $f(x_0,\cdot)$ is only strictly convex.
  Moreover, since MFCQ holds at
  $(x_0,u^*(x_0))$ (because SC holds), by~\citep[Remark 3.6]{TR:84}
  the constraint set is pseudo-Lipschitzian~\citep[Definition
  1.1]{NDY:95}. All of this implies by~\citep[Theorem 2.1]{NDY:95}
  that $u^*$ has the H\"older property at~$x_0$.
    
  Finally,~\ref{it:third} follows from the fact that SC implies MFCQ
  and~\citep[Theorem 1]{DR-SD:95}. Note that in this case, the 
  assumptions of~\citep[Theorem 1]{DR-SD:95}
  are satisfied by only requiring that $f(x,\cdot)$ is strictly convex for all $x\in\real^n$,
  instead of strongly convex for all $x\in\real^n$.
\end{proof}

In Proposition~\ref{prop:regularity}, note that neither~\ref{it:first}
implies~\ref{it:second} nor the converse. Even though the parametric
optimizer in Robinson's counterexample is not locally Lipschitz,
Proposition~\ref{prop:regularity} shows that it enjoys other, slightly
weaker, regularity properties. In particular, this result implies that
$u^*_4$, the fourth component of the parametric optimizer of
Robinson's counterexample, is continuous, cf. Figure~\ref{fig:sim}.

\begin{figure}[htb]
  \centering
  {\includegraphics[width=0.95\linewidth]{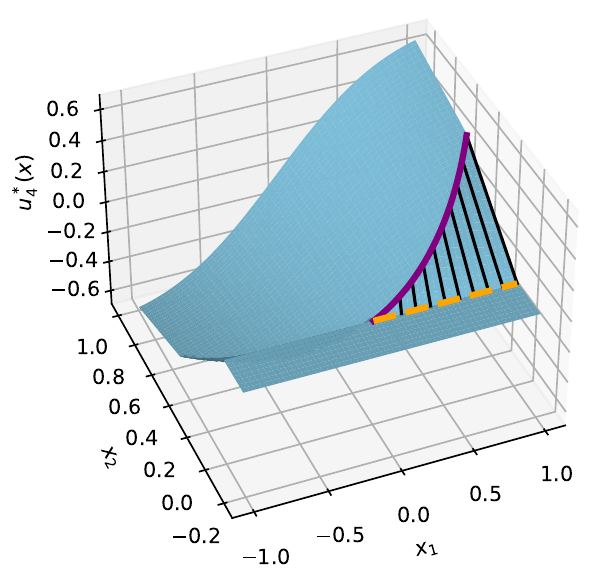}}
  \caption{Surface plot of $u_4$, which is the fourth component
    of the parametric optimizer of Robinson's counterexample,
    cf.~\eqref{eq:robinsons-counterexample}. The purple solid
      and orange dashed lines correspond to the plot of $u_4(p(s))$
      and $u_4(q(s))$ (defined in Example \ref{ex:robinson-counterex})
      respectively, for $s \geq 0$. The solid black lines connect
      points corresponding to $p(s)$ and $q(s)$.  The plot shows that
    $u_4$ is continuous at the origin, in agreement with
    Proposition~\ref{prop:regularity}. However, since
      the slope of the line connecting $u_4(p(s))$ and $u_4(q(s))$
      becomes arbitrarily large as $p$ and $q$ approach the origin,
      $u_4$ is not locally Lipschitz at the origin.}
  \label{fig:sim}
\end{figure}
%
%

Proposition~\ref{prop:regularity} also clarifies a confusion that has
arisen in the literature due to the loose use of terminology. Indeed,
according to~\citep[Theorem 6.4]{GS:18}, a parametric optimization
problem whose data satisfies the properties of Robinson's
counterexample has a Lipschitz minimizer!  This apparent contradiction
is rooted in different notions of Lipschitzness. Indeed, the notion of
Lipschitzness used in~\citep[Theorem 6.4]{GS:18} corresponds to
point-Lipschitzness.

Next we show that in the special case of parametric quadratic programs
that satisfy the assumptions of Proposition~\ref{prop:regularity} with
a scalar optimization variable, the parametric optimizer is locally
Lipschitz.

\begin{proposition}\longthmtitle{Scalar parametric quadratic programs
    have locally Lipschitz optimizers}\label{prop:scalar-qp-lipschitz}
  Suppose that $f\in\Cc^{2}(\real^n\times \real^1 )$ and and for
  $i\in\{1,\hdots,p \}$, let $g_i^0:\real^n\to\real$,
  $g_i^1:\real^n\to\real$ belong to $\Cc^2(\real^n)$ and
  \begin{align*}
    g(x,u) =
    \begin{pmatrix}
      g_1^0(x)u + g_1^1(x)
      \\
      \vdots
      \\
      g_i^0(x)u + g_i^1(x)
      \\
      \vdots
      \\
      g_p^0(x)u + g_p^1(x)
    \end{pmatrix}.
  \end{align*}
  Further assume that for any $x\in\real^n$, $f(x,\cdot)$ is strictly
  convex and~\eqref{eq:optim-based-controller} has at least one minimizer. Suppose that SC holds at $x_0\in\real^n$.  Then, $u^*$ is
  locally Lipschitz at $x_0$.
\end{proposition}
\begin{proof}
  First, since~\eqref{eq:optim-based-controller} has at least one minimizer
  for all $x\in\real^n$, the convexity assumptions of $f$ and $g$ imply that 
  $u^*$ is a singleton for all $x\in\real^n$.
  Note that for all $i\in\Ic(x_0,u^*(x_0))$, $g_i^0(x_0)\neq 0$.
  Indeed, if $g_i^0(x_0) = 0$ and $i\in\Ic(x_0,u^*(x_0))$, it follows
  that $g_i^1(x_0) = 0$, which implies that Slater's condition at
  $x_0$ is violated. Hence, $g_i^0(x_0)\neq 0$ for all
  $i\in\Ic(x_0,u^*(x_0))$ and the CR holds at
  $(x_0,u^*(x_0))$. Moreover, since Slater's condition holds at $x_0$,
  by~\citep[Prop. 5.39]{NA-AE-MP:20}, since $g(x_0,\cdot)$ is convex,
  MFCQ holds at $(x_0,u^*(x_0))$.  By~\citep[Theorem 3.6]{JL:95}, this
  implies that $u^*$ is locally Lipschitz at $x_0$.
\end{proof}

Note that Robinson's counterexample or~\citep[Example 3.11]{JL:95} do
not contradict Proposition~\ref{prop:scalar-qp-lipschitz}, since in
those two examples the optimization variable of the quadratic program
has dimensions four and two, respectively.
Proposition~\ref{prop:scalar-qp-lipschitz} shows that
optimization-based controllers for single-input systems with affine
constraints (e.g., obtained from CBF or CLF based conditions for
control-affine systems) that satisfy Slater's conditions are locally
Lipschitz.

The following examples show that the results from
Proposition~\ref{prop:regularity} do not hold if the assumptions are
weakened, even slightly.

\begin{example}\longthmtitle{Not point-Lipschitz optimizer without
    differentiability of problem data with respect to the
    parameter}\label{ex:no-PL-without-diff-wrt-parameter} {\rm If $f$
    and $g$ are not differentiable with respect to the parameter $x$
    but the rest of the assumptions of
    Proposition~\ref{prop:regularity} hold (even with $f(x,\cdot)$ strongly convex 
    for all $x\in\real^n$), 
    the following example,
    inspired by Robinson's counterexample, shows that the parametric
    optimizer is not necessarily point-Lipschitz.  Let
    $x=(x_1, x_2)\in\real^2$ and
    consider~\eqref{eq:robinsons-counterexample} with
    \begin{align*}
      A(x) =
      \begin{bmatrix}
        0 & -1 & 1 & 0, \\
        0 & 1 & 1 & 0, \\
        -1 & 0 & 1 & 0, \\
        1 & 0 & 1 & \sqrt{|x_1|} \end{bmatrix} , \quad b(x)
      =
      \begin{bmatrix}
        1 \\
        1 \\
        1 \\
        1+x_2
      \end{bmatrix}.
    \end{align*}
    Let $\tilde{u}^*:\real^2\to\real^4$ be its parametric solution and
    let $\tilde{u}_4^*:\real^2\to\real$ denote its fourth component,
    which is given by
    \begin{align*}
      \tilde{u}_4^*(x)=
      \begin{cases}
        0
        &\text{if} \ x_2\leq 0,
        \\
        \frac{x_2}{\sqrt{|x_1|}} \
        &\text{if} \ x_2\geq0, \ x_1\neq0,
          \frac{|x_1|}{2}\geq x_2,
        \\ 
        \frac{\sqrt{|x_1|}(x_2+1)}{|x_1|+2} \
        &\text{otherwise}.
      \end{cases}
    \end{align*}
    Let $x_1>0$ and define $p_{x_1}=(x_1,\frac{|x_1|}{2})$. Note that
    \begin{align*}
      \frac{\norm{\tilde{u}_4^*(p_{x_1})-\tilde{u}_4^*(0)}}{\norm{p_{x_1}
      - 0}} = \frac{1}{\sqrt{5 |x_1|}}. 
    \end{align*} 
    Since $x_1$ can be taken to be arbitrarily small, $\tilde{u}^*$ is
    not point-Lipschitz at the origin.  However, because $f$ and $g$,
    as well as their first and second derivatives with respect to $u$,
    are continuous in $u$ and $x$, and the rest of assumptions of
    Proposition~\ref{prop:regularity} hold (with $f(x,\cdot)$ strongly convex for all $x\in\real^n$), then by~\citep[Theorem
    5.3]{AVF-JK:85}, the corresponding parametric optimizer, and hence
    $\tilde{u}_4^*$, is continuous. \demo}
\end{example}

\begin{example}\longthmtitle{Discontinuous optimizer without Slater's
    condition}\label{ex:discontinuous-no-SC} {\rm The following
    example, taken from~\citep[Section VI]{BJM-MJP-ADA:15}, shows that
    if Slater's condition does not hold, then continuity of the
    parametric optimizer is not guaranteed even if the rest of
    assumptions from Proposition~\ref{prop:regularity} (even with $f(x,\cdot)$ strongly convex for all $x\in\real^n$) do hold:
    \begin{subequations}\label{eq:ames-counterex}
      \begin{align}    
        \hat{u}^*(x)
        &= \underset{u\in\real}{\text{argmin}}
          \frac{1}{2} u^2 - 2u,
        \\ 
        &\qquad \text{s.t.} \quad x u \leq 0.
      \end{align}
    \end{subequations}
    Indeed, the objective function and constraint
    of~\eqref{eq:ames-counterex} are twice 
    continuously differentiable, the objective function is strongly
    convex and the constraint is convex 
    for any $x\in\real$. However, Slater's condition does not hold at
    $x=0$. In fact,
    \begin{align*}
      \hat{u}^*(x) =
      \begin{cases}
        2 \quad \text{if} \quad x \leq 0,
        \\
    0 \quad \text{else},
      \end{cases}
    \end{align*}
    is discontinuous at $x=0$.  However, note that $\hat{u}^*$ is
    bounded.  \demo}
\end{example}

\begin{example}\longthmtitle{Unbounded optimizer without Slater's
    condition}\label{ex:unbounded-no-SC} {\rm The following example,
    adapted from~\citep[Example III.5]{MA-NA-JC:23-tac}, shows that if
    Slater's condition fails, not only can the parametric optimizer
    fail to be continuous, as shown in
    Example~\ref{ex:discontinuous-no-SC}, but it can even fail to be
    locally bounded. Let $x=(x_1,x_2)\in\real^2$,
    $a(x)=2x_1 x_2 + x_2^2(1-x_1^2-x_2^2)$, and consider:
  \begin{subequations}\label{eq:unbounded-counterex}
    \begin{align}
      \breve{u}^*(x)
      &= \underset{u\in\real}{\text{argmin}} \frac{1}{2}
        \norm{u}^2,
      \\ 
      &\qquad \text{s.t.} \quad a(x) + 2x_2^3 u \leq 0. 
    \end{align}
  \end{subequations}
  Note that Slater's condition does not hold at the point $x=(1,0)$.
  Moreover, $\breve{u}^*$ is given by:
  \begin{align*}
    \breve{u}^*(x) =
    \begin{cases}
      0 \quad \qquad \text{if} \quad a(x) \leq 0,
      \\
      -\frac{a(x)}{2x_2^3} \quad \text{else}.
    \end{cases}
  \end{align*}
  Note that $a(1,0)=0$ and $a(1,\epsilon)=2\epsilon - \epsilon^4$.
  Moreover, since any 
  neighborhood of $(1,0)$ contains points of the form $(1,\epsilon)$ 
  for sufficiently 
  small $\epsilon>0$, for any neighborhood $\Nc$ of $(1,0)$ 
  there exists $\epsilon_{\Nc}>0$ sufficiently small such that 
  $a(1,\epsilon_{\Nc})>0$. Now, since 
  \begin{align*}
    \lim\limits_{ \epsilon\to 0 } \frac{a(1,\epsilon)}{2\epsilon^3} = \infty,
  \end{align*}
  and $\breve{u}^*(x) = -\frac{a(x)}{2x_2^3}$ if $a(x)>0$,
  it follows that $\breve{u}^*$ is not locally bounded.  \demo}
\end{example}

Discontinuous controllers are relevant, and even necessary, in
multiple applications, cf.~\citep{JC:08-csm}. When dealing with
discontinuous systems, one needs to ensure basic properties such as
local boundedness and measurability. In the following, we provide
conditions that guarantee these properties for optimization-based
controllers.

The following result gives a condition which ensures that parametric
optimizers are locally bounded, hence precluding the behavior
exhibited in Example~\ref{ex:unbounded-no-SC}.

\begin{proposition}\longthmtitle{Conditions for local
    boundedness}\label{prop:suff-cond-local-boundedness}
  Suppose that $f$ and $g$ belong to
  $\Cc^{0}(\real^n\times\real^m)$. Further assume that for any
  $x\in\real^n$, $f(x,\cdot)$ is strictly convex, $g(x,\cdot)$ is
  convex, and~\eqref{eq:optim-based-controller} has at least one minimizer.
  Then, given
  $x_0\in\real^n$, $u^*$ is locally bounded at $x_0$ if and only if
  LCF holds at $x_0$.
\end{proposition}
\begin{proof}
  Note that since~\eqref{eq:optim-based-controller} has at least one minimizer
  for all $x\in\real^n$, the convexity assumptions on $f$ and $g$ imply that $u^*$ is a singleton 
  for all $x\in\real^n$.
  First suppose that LCF holds at $x_0$.  Therefore, there exists a
  compact set $K\subset\real^m$ and $\delta>0$ such that for all
  $y\in\real^n$ such that $\norm{y-x} < \delta$, there exists $u\in K$
  such that $g(y,u)\leq0$.  Since $f$ is continuous and $K$ is
  compact, there exists $B_f>0$ such that $|f(y,u)|<B_f$ for all
  $u\in K$ and $y\in\real^n$ such that $\norm{y-x_0}<\delta$.  Since
  for all $y\in\real^n$ such that $\norm{y-x_0}<\delta$, there exists
  a feasible $u\in K$, it follows that $|f(y,u^*(y))|<B_f$ for all
  $y\in\real^n$ such that $\norm{y-x_0}<\delta$. This implies that
  $u^*$ is locally bounded at $x_0$.  Now suppose that $u^*$ is
  locally bounded at $x_0$ and suppose, by contradiction, that LFC
  does not hold at $x_0$. Then, for any $\delta>0$ and compact set
  $K$, there exists $y\in\real^n$ with $\norm{y-x}\leq \delta$ and
  such that all $u\in\real^m$ with $g(y,u)\leq 0$ satisfy $u\notin K$.
  This means that there exists a sequence
  $\{ y_n \}_{n\in\mathbb{Z}_{>0}}$ such that $\norm{y_n-x}\leq 1/n$
  and $\norm{u^*(y_n)} \geq n$ for all $n\in\mathbb{Z}_{>0}$, which
  implies that $u^*$ is not locally bounded, hence reaching a
  contradiction.
\end{proof}

Verifying the local compact feasibility property can be challenging in
general.  However, for the particular case of CBF-based quadratic
programs,~\citep[Theorem V.1]{MA-NA-JC:23-tac} gives an alternative
sufficient condition for local boundedness of $u^*$ 
that only requires solving a specific linear equation.

Next, we turn our attention to the measurability properties of~$u^*$.

\begin{proposition}\longthmtitle{Sufficient conditions for
    measurability}\label{prop:suff-cond-meas}
  Suppose that $f$ and $g$ belong to
  $\Cc^{0}(\real^n\times\real^m)$. Further assume that for any
  $x\in\real^n$, $f(x,\cdot)$ is strictly convex, $g(x,\cdot)$ is
  convex, and~\eqref{eq:optim-based-controller} has at least one minimizer.
  Further assume
  that for every $x\in\real^n$, LCF holds at $x$.
  Then, $u^*$ is measurable.
\end{proposition}
\begin{proof}
  Note that since~\eqref{eq:optim-based-controller} has at least one minimizer
  for all $x\in\real^n$, the convexity assumptions on $f$ and $g$ imply that $u^*$ is a singleton 
  for all $x\in\real^n$.
  We use the Measurable Maximum Theorem~\citep[Theorem
  18.19]{CDA-KCB:99}. We assume that $\real^n$ and $\real^m$ are
  equipped with the usual Borel $\sigma$-algebras. Since $f$ is continuous, 
  it is a Carath\'{e}odory
  function (cf.~\citep[Definition 4.50]{CDA-KCB:99}). Therefore, we
  only need to ensure that the set-valued map
  $\phi: x \to \setdef{u\in\real^m}{g(x,u)\leq 0}$ is a weakly
  measurable correspondence (cf.~\citep[Definition 18.1]{CDA-KCB:99})
  with nonempty compact values.  The fact that $\phi$ takes nonempty
  values follows from the fact that the feasible set
  $\setdef{u\in\real^m}{g(x,u)\leq 0}$ is nonempty. Moreover, since
  Proposition~\ref{prop:suff-cond-local-boundedness} ensures that
  $u^*$ is locally bounded at every $x\in\real^n$, without loss of
  generality we can assume that $\phi$ takes compact values
  (otherwise, we can define extra constraints that ensure that the
  feasible set is bounded for every $x\in\real^n$ without changing the
  optimizer $u^*$).  Now, to show that $\phi$ is a weakly measurable
  correspondence, we follow an argument similar to the proof
  of~\citep[Corollary 18.8]{CDA-KCB:99}. For every
  $n\in\mathbb{Z}_{>0}$, define the set-valued map
  $\phi_n: x \to \setdef{u\in\real^m}{g(x,u)\leq 1/n}$.  By
  Lemma~\citep[Corollary 18.8]{CDA-KCB:99}, $\phi_n$ is measurable.
  Moreover, for every $x\in\real^n$ and $n\in\mathbb{Z}_{>0}$,
  $\phi(x)\subset \partial(\phi_n(x))$ (where $\partial(\phi_n(x))$
  denotes the boundary of $\phi_n(x)$), and
  $\phi(x) = \cap_{n=1}^{\infty} \partial(\phi_n(x))$.  Furthermore,
  again without loss of generality, $\partial(\phi_n)$ has compact
  values for every $n\in\mathbb{Z}_{>0}$ (again, if that is not the
  case we can define extra constraints that ensure that this holds 
  without changing the optimizer $u^*$), and by~\citep[Theorem
  18.4(3)]{CDA-KCB:99}, the intersection
  $\phi: x\to \cap_{n=1}^{\infty} \partial(\phi_n(x))$ is measurable.
\end{proof}

The second column in Table~\ref{tab:1} summarizes the different
results discussed in this section.

\begin{remark}\longthmtitle{Verifying constraint qualifications and conditions
  in practice without knowledge of the
  optimizer}\label{rem:practical-considerations} 
\rm{To show that $u^*$ is locally Lipschitz at a point $x\in\real^n$
  using ~\citep[Theorem 3.6]{JL:95}, we need to verify that both MFCQ
  and CR hold at $(x,u^*(x))$. Similarly,~\citep[Theorem 4.1]{SMR:80}
  (resp.~\citep[Theorem 2.1]{AVF:76}) require the verification of LICQ
  (resp. LICQ and SCS) at $(x,u^*(x))$. These results require
  knowledge of $u^*(x)$ to verify the corresponding property holds
  at~$x$. However, in several applications it can be useful to know
  the regions where the controller $u^*$ is discontinuous (for
  instance, to design safety-critical controllers that avoid such
  regions).  Slater's condition is useful for this purpose because
  Proposition~\ref{prop:regularity} guarantees different regularity
  properties of $u^*$ at $x$ without requiring knowledge of $u^*(x)$
  (assuming that the extra conditions on differentiability and
  convexity of the objective function and constraints in
  Proposition~\ref{prop:regularity} also hold).  Moreover, in the
  special case where the constraints
  in~\eqref{eq:optim-based-controller} are affine, i.e.,
    \begin{align*}
      g(x,u) =
      \begin{pmatrix}
        g_1^0(x)^T u + g_1^1(x) \\
        \vdots \\
        g_i^0(x)^T u + g_i^1(x) \\
        \vdots \\
        g_p^0(x)^T u + g_p^1(x)
      \end{pmatrix},
    \end{align*}
    for $i\in\{1,\hdots,p\}$ and $g_i^0:\real^n\to\real^m$ and
    $g_i^1:\real^n\to\real$ in $\Cc^2(\real^n)$, then~\citep{WBC:22}
    shows that by letting $c_x^*$ be the optimal value of the linear
    program
    \begin{subequations}
      \begin{align}
        &\max\limits_{u\in\real^m} \sum_{i=1}^m |u_i| \\
        &\text{s.t.} \quad u_i \geq 0, \ i\in\{1,\hdots,m\}, \\
        &\qquad \sum_{i=1}^m u_i g_i^0(x) = 0, \\
        &\qquad \sum_{i=1}^m u_i g_i^1(x) = 0.
      \end{align}
      \label{eq:linear-program-ustar}
    \end{subequations}
    then Slater's condition holds at $x$ if and only if
    $c_x^*=0$. Hence,~\eqref{eq:linear-program-ustar} can be solved
    before solving~\eqref{eq:optim-based-controller} to verify that
    $u^*$ satisfies the regularity properties in
    Proposition~\ref{prop:regularity}.  \demo }
\end{remark}

\section{Existence and Uniqueness of Solutions under
  Optimization-Based
  Controllers}\label{sec:existence-uniqueness-conditions}

In this section, we leverage the regularity properties established in
Section~\ref{sec:regularity-properties} to study existence and
uniqueness of solutions for the closed-loop
system~\eqref{eq:closed-loop} under the optimization-based
controller~$u^*$.

First, we note that by the Picard-Lindel\"of theorem~\citep[Theorem
2.2]{GT:12-ams}, any of the assumptions described in
Section~\ref{sec:regularity-properties} that guarantee that $u^*$ is
locally Lipschitz at a point $x_0$ also guarantee that the closed-loop
system~\eqref{eq:closed-loop} has a unique solution with initial
condition at $x_0$ for sufficiently small times.

The following result establishes existence of solutions under weaker
assumptions.

\begin{proposition}\longthmtitle{Existence of classical solutions for the
    closed-loop system}~\label{prop:existence-sols}
  Suppose that $f$ and $g$ belong to
  $\Cc^{0,2}(\real^n\times\real^m)$. Further assume that for any
  $x\in\real^n$, $f(x,\cdot)$ is strictly convex, $g(x,\cdot)$ is
  convex, and~\eqref{eq:optim-based-controller} has at least one minimizer. 
  Further assume that SC holds at $x_0\in\real^n$.  
  Let
  $F:\real^n\times\real^m\to\real^n$ be locally Lipschitz. Then,
  there exists $\delta_0 > 0$ such that the differential equation
  \eqref{eq:closed-loop}
  has at least one solution
  $x:(-\delta_0,\delta_0)\to\real^n$ with initial condition $x(0)=x_0$.
\end{proposition}
\begin{proof} Note that
    since~\eqref{eq:optim-based-controller} has at least one minimizer
    for all $x\in\real^n$, the convexity assumptions on $f$ and $g$
    imply that $u^*$ is a singleton for all $x\in\real^n$.
    Since SC holds at $x_0$,
    by~\citep[Prop. 5.39]{NA-AE-MP:20}, since $g(x_0,\cdot)$ is
    convex, MFCQ holds at $(x_0,u^*(x_0))$.  By~\citep[Theorem
    5.3]{AVF-JK:85}, this implies that $u^*$ is continuous at $x_0$.
    Since $g$ is continuous, there exists a neighborhood $\Vc_{x_0}$
    of $x_0$ such that SC holds for all $y\in\Vc_{x_0}$. By the same
    argument, this implies that $u^*$ is continuous in $\Vc_{x_0}$.
    The result now follows by Peano's existence theorem~\citep[Theorem
    2.1]{EAC-NL:55}.
\end{proof}

Proposition~\ref{prop:existence-sols} implies in
  particular that, under the assumptions of
  Proposition~\ref{prop:regularity}, the closed-loop
  system~\eqref{eq:closed-loop} has at least one solution in a
  neighborhood of $x_0$.

Next, we study uniqueness of solutions under the assumptions of
Proposition~\ref{prop:regularity}. We first note that the H\"older
property does not imply uniqueness, even in simple one-dimensional
examples. For example, the differential equation $\dot{x}=x^{1/3}$ has
the H\"older property at $0$ but infinitely many solutions starting
from the origin.  The next example shows that, in general,
point-Lipschitz continuity does not imply uniqueness of solutions
either.

\begin{example}\longthmtitle{Point-Lipschitz differential equation
    with non-unique solutions}\label{ex:point-lip-nonunique}
  {\rm Let $u^*:\real^2\to\real^4$ be the parametric optimizer of
    Robinson's counterexample.  Consider the dynamical system
    \begin{subequations}\label{eq:counterexample-point-lip-nonuniqueness}
      \begin{align}
        \dot{x}_1 &= \frac{1}{2},
        \\
        \dot{x}_2 &= u^*_4(x_1,x_2) ,
      \end{align}
    \end{subequations}
    with initial condition $(x_1(0),x_2(0))=(0,0)$.
    By Proposition~\ref{prop:regularity}, the vector field
    in~\eqref{eq:counterexample-point-lip-nonuniqueness} is
    point-Lipschitz at the origin. However,~\eqref{eq:counterexample-point-lip-nonuniqueness} 
    admits the
    following two distinct solutions starting from the origin:
    $y_1(t):=(\frac{1}{2}t,0)$ and
    $y_2(t):=(\frac{1}{2}t,\frac{1}{8}t^2)$,
    cf. Figure~\ref{fig:counterex-vector-field}.  \demo }
\end{example}

\begin{figure}[htb]
  \centering
  {\includegraphics[width=0.99\linewidth]{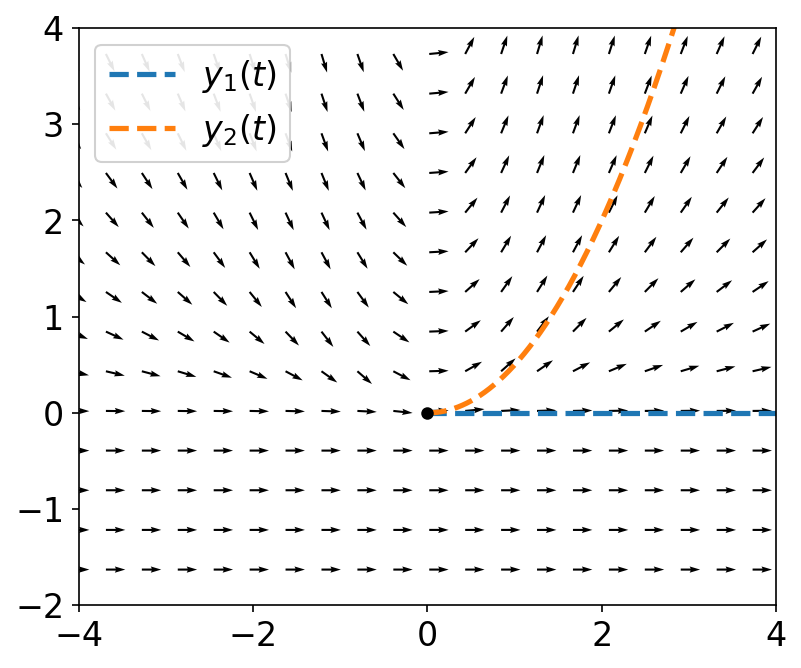}}
  \caption{The arrows depict the vector
    field~\eqref{eq:counterexample-point-lip-nonuniqueness}. The
    dashed blue and orange curves depict the two solutions $y_1$ and
    $y_2$ starting from the origin, where the vector field is
    point-Lipschitz but not locally Lipschitz.}
  \label{fig:counterex-vector-field}
\end{figure}

Hence, in general the assumptions of
Proposition~\ref{prop:regularity} are not sufficient to ensure
uniqueness of solutions of the closed-loop system.  Interestingly, the
next result shows that point-Lipschitz continuity guarantees
uniqueness of solutions starting from equilibria.

\begin{proposition}\longthmtitle{Point-Lipschitz continuity and
    uniqueness}\label{prop:point-lipschitzness-uniqueness}
  Let $\tilde{F}:\real^n\to\real^n$ be point-Lipschitz at
  $x_0 \in \real^n$ and $\tilde{F}(x_0)=0$. Then
  the function $x(t)=x_0$ for all $t\geq 0$ is the unique solution to 
  the differential equation
  $\dot{x}=\tilde{F}(x)$ with initial condition $x(0)=x_0$. 
\end{proposition}
\begin{proof}
  Let $\delta > 0$ and $L$ be the point-Lipschitz continuity constant
  of $\tilde{F}$ and take $\delta < \frac{1}{L}$.  Suppose that there
  exists another solution $y:[0,\delta)\to\real^n$ starting from
  $x_0$. Then, $\sup_{t\in[0,\delta)}\norm{y(t)-x_0} > 0$. Moreover,
  \begin{align*}
    &\sup\limits_{t\in[0,\delta)}\norm{y(t)-x_0} =
      \sup\limits_{t\in[0,\delta)}\norm{\int_0^t \tilde{F}(y(s))ds} =
    \\
    &\sup\limits_{t\in[0,\delta)}\norm{\int_0^t \Big(
      \tilde{F}(y(s))-\tilde{F}(x_0) \Big) ds} \leq
    \\
    &\sup\limits_{t\in[0,\delta)} \int_0^t L\norm{y(s)-x_0}ds
      \leq 
    \\
    & L\delta \sup_{t\in[0,\delta]}\norm{y(t)-x_0}
      <\sup_{t\in[0,\delta]}\norm{y(t)-x_0}
  \end{align*}
  where in the last inequality we have used the fact that
  $\sup_{t \in [0,\delta)}\norm{y(t)-x_0} > 0$. We hence reach a
  contradiction, which means that the constant solution is the only
  solution for $t\in[0,\delta)$. By repeating the same argument at
  time $\delta$, we can extend this constant solution for all positive
  times.
\end{proof}

This result implies that in one dimension, point-Lipschitz ODEs have
unique solutions.

\begin{corollary}\longthmtitle{Point-Lipschitz continuity implies uniqueness
    in one dimension}\label{cor:pl-1d-uniqueness}
  Let $\tilde{F}:\real\to\real$ be continuous in a neighborhood of
  $x_0$ and point-Lipschitz at $x_0$.  Then, the differential equation
  $\dot{x}=\tilde{F}(x)$ with initial condition $x(0)=x_0$ has a
  unique solution.
\end{corollary}
\begin{proof}
  If $\tilde{F}(x_0)\neq0$, by~\citep[Theorem~1.2.7]{RPA-VL:93}, the
  differential equation has only one solution. If $\tilde{F}(x_0)=0$,
  the result follows from
  Proposition~\ref{prop:point-lipschitzness-uniqueness}.
\end{proof}

If Slater's condition does not hold but the rest of assumptions of
Proposition~\ref{prop:regularity} hold
(even with $f(x,\cdot)$ strongly convex for all $x\in\real^n$),
Example~\ref{ex:discontinuous-no-SC} shows that $u^*$ can be
discontinuous, in which case neither existence nor uniqueness of
solutions is guaranteed. 
In the case where $f$ and $g$ are not
differentiable with respect to the parameter, but the rest of
assumptions of Proposition~\ref{prop:regularity} hold 
(even with $f(x,\cdot)$ strongly convex for all $x\in\real^n$),
Example~\ref{ex:no-PL-without-diff-wrt-parameter} shows that $u^*$ is
continuous but not necessarily point-Lipschitz. Therefore, in this
case existence is guaranteed but uniqueness is not.

Note that so far, we have only considered existence and uniqueness for
classical solutions.  For discontinuous dynamical systems, other
notions of solution, such as Carath\'{e}odory or Filippov solutions,
can be defined, cf.~\citep{JC:08-csm}.  The following result gives
conditions on~\eqref{eq:optim-based-controller} that ensure the
existence of Filippov solutions for~\eqref{eq:closed-loop}.

\begin{proposition}\longthmtitle{Existence of Filippov
    solutions for the closed-loop system}\label{prop:existence-filippov}
  Suppose that $f$ and $g$ belong to
  $\Cc^{2}(\real^n\times\real^m)$.  Further assume that for any
  $x\in\real^n$, $f(x,\cdot)$ is strictly convex,
  $g(x,\cdot)$ is
  convex and~\eqref{eq:optim-based-controller} has at least one minimizer.  Finally suppose that for
  every $x\in\real^n$, LCF holds at $x$ and
  $F:\real^n\times\real^m\to\real^n$ is locally Lipschitz.  Then for
  any $x\in\real^n$, there exists $\delta_x>0$ such
  that~\eqref{eq:closed-loop} has at least one Filippov solution
  $y:(-\delta_x,\delta_x)\to\real^n$ with initial condition $y(0)=x$.
\end{proposition}
\begin{proof}
  Note that since~\eqref{eq:optim-based-controller} has at least one minimizer
  for all $x\in\real^n$, the convexity assumptions on $f$ and $g$ imply that $u^*$ is a singleton 
  for all $x\in\real^n$.
  By Propositions~\ref{prop:suff-cond-local-boundedness}
  and~\ref{prop:suff-cond-meas}, the assumptions ensure that $u^*$ is
  measurable and locally bounded.  The result follows
  from~\citep[Theorem 7]{AFF:88}.
\end{proof}

A weaker condition to ensure that Filippov solutions are unique is
that the closed-loop system~\eqref{eq:closed-loop} is
\emph{essentially one-sided Lipschitz}~\citep{JC:08-csm}.  Although it
is known how to verify this property for projected dynamical systems
(see e.g., \citep[proof of Theorem~2.7]{AN-DZ:96} for the Euclidean
case, and \citep[Proposition~6.12]{AH-SB-FD:21} for the Non-Euclidean
case), to the best of our knowledge there exist no results in the
parametric optimization literature that guarantee that the parametric
optimizer $u^*$, or the closed-loop dynamics, satisfies this property.

Finally, we note that if $f$ is not strictly convex or $g$ is not
convex, the optimizer $u^*$ is not guaranteed to be single-valued,
which means that the usual notions of regularity of the controller and
of solutions of the closed-loop system are not well defined. The third
and fourth columns of Table~\ref{tab:1} summarize the results
presented in this section.

\section{Forward Invariance Properties of Optimization-Based Controllers}
\label{sec:forward-invariance}
In this section we study conditions that guarantee the forward
invariance of a set for the closed-loop system under an
optimization-based controller.  Recall the notion of
\textit{tangent cone} to a set $\Cc\subset\real^n$:


The basic result concerning forward invariance is the following:
\begin{theorem}\longthmtitle{Nagumo's Theorem~\citep{MN:42,FB:99}}\label{thm:nagumos-theorem}
  Let $\tilde{F}:\real^n\to\real^n$ and consider the system
  $\dot{x} = \tilde{F}(x)$. Assume that, for each initial condition in
  a set $\Dc \subset \real^n$, it admits a unique forward complete
  solution (i.e., a unique solution defined for all positive times).
  Let $\Cc \subset \Dc \subset \real^n$ be a closed set. Then the set
  $\Cc$ is forward invariant for the system if and only if
  $\tilde{F}(x) \in T_{\Cc}(x)$ for all $x\in\Cc$ (here, $T_{\Cc}(x)$
  is the \textit{tangent cone}\footnote{Recall that
    $ T_{\Cc}(x) = \bigl\{ v\in\real^n \ \big\rvert \
    \liminf\limits_{h\to0} \frac{\textnormal{dist}(x+hv,\Cc)}{h} = 0
    \bigl\}$.} to $\Cc\subset\real^n$ at $x\in\real^n$).
\end{theorem}
\smallskip

The condition that $\tilde{F}(x)\in T_{\Cc}(x)$ for all $x\in\Cc$ is
called the \textit{sub-tangentiality condition}, and can be enforced
using the constraints of an optimization-based feedback controller of
the form \eqref{eq:optim-based-controller}.  We show how in the
following.  Suppose that $\Cc$ is parameterized as
$\Cc = \{x \in \real^n \mid h_j(x) \geq 0,\; 1 \leq j \leq p \}$,
where $h_j:\real^n \to \real$ are continuously differentiable for
$j=1, \dots, p$, and the dynamics take the form
\begin{equation}
  \label{eq:control-affine}
  \dot{x} = F(x, u) = F_0(x) + \sum_{i=1}^{m}u_iF_i(x) ,
\end{equation}
for smooth functions $F_i:\real^n \to \real^n$ for
$i\in\fromto{0}{m}$. Next, define $A(x) \in \real^{p \times m}$ and
$b(x) \in \real^p$ as
\[
  A(x) =
  \begin{bmatrix}
    \Lc_{F_1}h_1(x) & \dots  & \Lc_{F_m}h_1(x) \\
    \vdots & \ddots & \vdots \\
    \Lc_{F_1}h_p(x) & \dots  & \Lc_{F_m}h_p(x)
  \end{bmatrix} ,
\]
\[
  b(x) =
  \begin{bmatrix}
    -\alpha(h_1(x)) - \Lc_{F_0}h_1(x) \\
    \vdots \\ 
    -\alpha(h_m(x)) - \Lc_{F_0}h_m(x)
  \end{bmatrix} ,
\]
where $\alpha$ is a class-$\Kc$ function.  Let $A_j(x)$ denote the
$j$th row of $A(x)$, and for $J \subset \until{p}$, let $A_J(x)$
denote the matrix consisting of the rows of $A(x)$ corresponding to
$j \in J$.

In the literature on optimization-based control design
\citep{ADA-XX-JWG-PT:17}, the feasibility of the system
$A_j(x)u \geq b_j(x)$ for all $x \in \real^n$ such that
$h_j(x) \geq 0$ is equivalent to $h_j$ being a \emph{control barrier
  function} for the set $\{ x \in \real^n \mid h_j(x) \geq 0\}$. Since
we are considering the case where $\Cc$ is possibly parameterized by
multiple inequalities, here we make the stronger assumption that the
system $A(x)u \geq b(x)$ (where the inequality holds component-wise)
is feasible for all~$x \in \Cc$.  In this case, if $\Cc$ satisfies an
appropriate constraint qualification condition (e.g., MFCQ or LICQ)
and $u^*:\real^n \to \real^m$ is a feedback controller such that
$A(x)u^*(x) \geq b(x)$ for all $x \in \Cc$, then the closed-loop
dynamics satisfies the sub-tangentiality condition
$F(x, u^*(x)) \in T_{\Cc}(x)$. Such a controller can be obtained from
the solution of a parametric optimization problem of the
form~\eqref{eq:optim-based-controller} where $g(x, u) = b(x) - A(x)u$.

To show invariance invoking Theorem~\ref{thm:nagumos-theorem}, one
needs to additionally ensure that the closed-loop dynamics has unique
solutions. The conditions discussed in Section
\ref{sec:existence-uniqueness-conditions} and summarized in Table
\ref{tab:1} can be translated into easily checkable conditions on the
objective function, the matrix $A(x)$, and the vector $b(x)$. The following result
uses~\citep[Theorem 3.6]{JL:95} to ensure uniqueness, and therefore
forward invariance.

\begin{theorem}\longthmtitle{Sufficient conditions for forward
    invariance with respect to closed-loop dynamics} 
  Consider the dynamics \eqref{eq:control-affine} and the
  optimization problem \eqref{eq:optim-based-controller} where
  $f \in \Cc^{1,2}(\real^{n} \times \real^m)$ is strictly convex, and
  $g(x, u) = b(x) - A(x)u$. Assume
  \begin{itemize}
  \item For all $x \in \Cc$, there exists $u \in \real^m$ such
    that~$A(x)u > b(x)$, and~\eqref{eq:optim-based-controller} has at least one minimizer.
  \item For all $x \in \Cc$, there is an open set
    $U_x \subset \real^n$ containing $x$ such that, for any subset
    $J$ in $\until{p}$, the matrix $A_J(y)$ has constant rank for
    all~$y \in U_x$.
  \end{itemize}
  Then the closed-loop system under the optimization-based
  controller~\eqref{eq:optim-based-controller} has unique solutions,
  and $\Cc$ is forward invariant.
\end{theorem}

In the case where the closed-loop dynamics are point-Lipschitz,
solutions are not necessarily unique and therefore forward invariance
of $\Cc$ cannot be guaranteed by Theorem~\ref{thm:nagumos-theorem}.
In fact, the following is an example of a system where the
sub-tangentiality condition holds but there exist solutions starting
in~$\Cc$ that eventually leave.

\begin{example}\longthmtitle{Point-Lipschitz differential equation
    violating forward invariance}\label{ex:pl-but-no-fwd-inv}
  \rm{Let $\Cc = \setdef{ (x_1,x_2)\in\real^2 }{ x_2 \leq 0 }$ and
    consider the system with the feedback controller defined in
    Example~\ref{ex:point-lip-nonunique}. Because
    $\Cc$ satisfies LICQ, the tangent cone can be computed as
    $T_{\Cc}(x_1,x_2)=\real^2$ if $x_2 < 0$, and
    $T_{\Cc}(x_1, 0)= \setdef{(\xi_1, \xi_2)}{\xi_2 \leq 0}$.  The
    closed-loop system satisfies
    $F(x, u^*(x)) = (\frac{1}{2}, u_4^*(x_1,x_2)) \in
    T_{\Cc}(x_1,x_2)$ for all $(x_1,x_2)\in\Cc$.  However, the
    solution $y_2(t)=(\frac{1}{2}t,\frac{1}{8}t^2)$ satisfies
    $y_2(0) \in \Cc$ and $y_2(t) \notin \Cc$ for all $t > 0$.}  \demo
\end{example}

Example~\ref{ex:point-lip-nonunique} is problematic
because it shows that even if the \textit{sub-tangentiality} condition
for a safe set $\Cc$ is included as one of the constraints of the
optimization-based controller, if the solutions of the closed-loop
system are not unique, some of the solutions might leave the safe
set~$\Cc$.
However, using the notion of minimal barrier
functions~\citep{RK-ADA-SC:21}, the following result gives a condition
for forward invariance that can be applied to systems with non-unique
solutions.

\begin{theorem}\longthmtitle{Minimal Barrier Functions,~\citep[Theorem
    1]{RK-ADA-SC:21}}\label{thm:minimal-bfs} 
  Let $\tilde{F}:\real^n\to\real^n$ be continuous and consider the
  system $\dot{x}=\tilde{F}(x)$.  Let $h:\real^n\to\real$ be a
  continuously differentiable function and let
  $\Cc=\setdef{x\in\real^n}{h(x)\geq0}$ be a nonempty set. If $h$ is a
  minimal barrier function, cf.~\citep[Definition 2]{RK-ADA-SC:21},
  then any solution of $\dot{x}=\tilde{F}(x)$ with initial condition
  in $\Cc$ remains in $\Cc$ for all positive times.
\end{theorem}

A simple scenario in which $h$ is a minimal barrier function is if
there exists a strictly increasing function $\alpha:\real\to\real$
with $\alpha(0)=0$ and an open set $\Dc$ with $\Cc\subset\Dc$ such
that $\nabla h(x)^\top \tilde{F}(x) \geq -\alpha(h(x))$ for all
$x\in\Dc$. Such a set $\Dc$ and class $\Kc$ function $\alpha$ do not exist
in Example~\ref{ex:pl-but-no-fwd-inv}. Since
Theorem~\ref{thm:minimal-bfs} only requires $\tilde{F}$ to be
continuous, the system $\dot{x}=\tilde{F}(x)$ might have multiple
solutions starting from the same initial condition. However, the
result ensures that if the initial condition is in $\Cc$, then all
solutions remain in $\Cc$ for all positive times.  Moreover, since
point-Lipschitz functions are continuous,
Theorem~\ref{thm:minimal-bfs} can be applied to differential equations
defined by point-Lipschitz functions.  Therefore, if one of the
constraints in~\eqref{eq:optim-based-controller} corresponds to the
minimal control barrier function condition
of a function $h$, and if the resulting controller is point-Lipschitz
(e.g., by satisfying the hypothesis of
Proposition~\ref{prop:regularity}), then all solutions of the
closed-loop system that start in
$\Cc:=\setdef{x\in\real^n}{h(x)\geq0}$ remain in $\Cc$ for all
positive times.

Finally, we also note that if $u^*$ is discontinuous, the closed-loop
system might not have unique solutions and hence the assumptions of
Theorem~\ref{thm:nagumos-theorem} will not hold.  Therefore, this
result cannot be used to guarantee forward invariance of sets.
However, the following result gives a sufficient condition for forward
invariance of sets under Filippov solutions.  It follows as an
adaptation of~\citep[Theorem 1]{MM-RGS:21}, which gives a sufficient
condition for forward invariance
of sets under hybrid systems.

\begin{theorem}\longthmtitle{Forward invariance under Filippov
    solutions of closed-loop dynamics}\label{thm:fwd-inv-filippov}
  Let $h:\real^n\to\real$ be a continuously differentiable function
  and $\Cc:=\setdef{x\in\real^n}{h(x)\geq0}$. Further let
  $\Dc\subset\real^n$ be a set containing $\Cc$ such that, for each
  initial condition $x_0$ in $\Dc$, there exists a forward complete
  Filippov solution of~\eqref{eq:closed-loop} with initial condition
  at $x_0$.  Let $\Pc(\real^n)$ denote the collection of subsets of
  $\real^n$ and let $\Fc:\real^n\to\Pc(\real^n)$ be the Filippov
  set-valued map of~\eqref{eq:closed-loop}, i.e.,
  \begin{align*}
    \Fc(x) := \bigcap_{\delta>0}\bigcap_{\mu(S)=0} \overline{co}
    \left\{ \bigcup_{y|\norm{y-x}\leq\delta} F(y,u^*(y))\backslash S
    \right\}, 
  \end{align*}
  where $\overline{co}$ denotes the convex closure and $\mu$ denotes
  the Lebesgue measure.  Further assume that there exists a
  neighborhood $\Uc_f$ of $\partial\Cc=\setdef{x\in\real^n}{h(x)=0}$
  such that
  \begin{align}\label{eq:filippov-fwd-invariance-condition}
    \nabla h(x)^T \eta \geq 0, \ \forall x\in \Uc_f\backslash\Cc \
    \text{and} \ \ \forall \eta\in\Fc(x). 
  \end{align}
  Then, all Filippov solutions of~\eqref{eq:closed-loop} with initial
  condition at $\Cc$ remain in $\Cc$ for all positive times.
\end{theorem}

In particular, Theorem~\ref{thm:fwd-inv-filippov} ensures that under
the assumptions of Proposition~\ref{prop:existence-filippov}, and if
Filippov solutions are defined for all positive times,
then the
\textit{sub-tangentiality}-like
condition~\eqref{eq:filippov-fwd-invariance-condition} guarantees
forward invariance of Filippov solutions.  We note also that
Theorem~\ref{thm:fwd-inv-filippov} is possibly conservative, and
tighter conditions that guarantee forward invariance for Filippov
solutions could be developed using an adapted notion of minimal
barrier functions for discontinuous systems.

\section{Conclusions and Outlook}

We have provided an integrative presentation of insights and results
about the regularity properties of optimization-based controllers, and 
their implication in different properties of interest of control systems.

\textit{Regularity properties:} Under appropriate constraint
qualifications and conditions on the data that defines the
optimization problem, we have shown that optimization-based
controllers are locally Lipschitz, continuously differentiable, and
even analytic.  We have also characterized the properties enjoyed by
parametric optimizers arising from optimization problems defined by
second-order continuously differentiable objective function and
constraints, strictly (or strongly)
%
%
convex objective
function, and feasible set with nonempty interior (the same properties
as in Robinson's counterexample).  We have shown that, even though
such parametric optimizers might not be locally Lipschitz, they enjoy
other important regularity properties, like point-Lipschitz
continuity.  Even if the optimization-based controller is
discontinuous, under appropriate conditions on the optimization
problem data, we have shown that it is measurable and locally bounded.

\textit{Implications on the resulting closed-loop systems:} When our
results are applied to the motivating examples in
Section~\ref{sec:introduction}, they improve upon the existing results
in the literature by providing a more detailed description of the
regularity of the corresponding controller under a wider range of
conditions.  Building on the results on regularity properties of
optimization-based controllers, we have studied the existence and
uniqueness of classical and Filippov solutions of closed-loop systems
obtained from an optimization-based controller, and identified
conditions ensuring that all (not necessarily unique) solutions remain
in a safe set of interest.

\textit{Outlook:} The results presented in this work show that the
regularity properties of optimization-based controllers are determined
by the smoothness/convexity and constraint qualification properties of
the optimization problems defining them. This opens the door to the
possibility of designing optimization problems with the appropriate
conditions and constraint qualification properties in order to endow
the associated optimization-based controller with certain desired
regularity properties.  For example, in the context of safety-critical
control, it is sufficient to find safe sets and control barrier
functions for those sets to guarantee that an associated control
barrier function based controller is locally Lipschitz. In certain
cases, one can guarantee that these control objectives are obtained
without continuity or even uniqueness of solutions to the resulting
closed-loop systems. Characterizing conditions on the objective
function and constraints to ensure that control objectives are
achieved even in the absence of local Lipschitz continuity is an
important direction for future work. In particular, understanding
one-sided Lipschitzness in the context of parametric optimization is
important in the context of safety-critical control, since uniqueness
of solutions allows one to verify forward-invariance of a safe set via
Nagumo's Theorem.  Besides the examples mentioned in
Section~\ref{sec:introduction}, the relevance of the results presented
herein also applies to other areas of systems and control, where
controllers need to be designed with desirable regularity properties,
such as backstepping~\citep{MK-IK-PK:95}, where virtual controllers
need to be differentiated at the intermediate layers to construct a
composite Lyapunov or barrier~\citep{MC-PO-GB-ADA:23} function for the
composite system.
Other ideas for future work include improving further the results
presented in this paper to more specific classes of problems, such as
CBF-based quadratic programs or second-order convex programs, as well
as MPC-based controllers. We also plan to relax the differentiability
and convexity assumptions considered throughout this paper and give
regularity results for set-valued optimizers using the theory of
continuous selections~\citep{EM:56}.

\section*{Acknowledgments}
This work was partially supported by ARL-W911NF-22-2-0231. The authors
would like to thank E. Dall'Anese and Y. Chen for multiple conversations on
optimization-based controllers, S. Ganguly for pointing out typos in
an earlier version of the manuscript, an anonymous reviewer of the
paper~\citep{PM-JC:24-auto} for various comments regarding Robinson's
counterexample, and M. Alyaseen for a discussion regarding
Example~\ref{ex:unbounded-no-SC}.

\bibliography{../bib/alias,../bib/JC,../bib/Main-add,../bib/Main}
\bibliographystyle{IEEEtran}

\end{document}